%% file: BaRu21_submission.tex
\documentclass[reqno]{amsart}
\usepackage{amsmath,amsthm}
\usepackage{dsfont} 
\usepackage[english]{babel}
\usepackage{color}
\usepackage{amssymb}
\usepackage{mathrsfs}
\usepackage{graphicx}
\usepackage[font=footnotesize, labelfont=bf]{caption}
\usepackage{mathtools}
\usepackage[colorlinks=true, pdfstartview=FitV, linkcolor=blue, citecolor=blue, urlcolor=blue,pagebackref=false]{hyperref}
\usepackage{microtype}
\usepackage{enumitem}
\usepackage[normalem]{ulem}
\usepackage{soul}

\usepackage{float}

\newtheorem{theorem}{Theorem}[section]
\newtheorem{proposition}[theorem]{Proposition}
\newtheorem{lemma}[theorem]{Lemma}
\newtheorem{corollary}[theorem]{Corollary}

\theoremstyle{definition}
\newtheorem{remark}[theorem]{Remark}
\newtheorem{example}[theorem]{Example}
\newtheorem{definition}[theorem]{Definition}

\newtheorem{assumption}{Assumption}
\numberwithin{equation}{section}

\definecolor{myblue}{RGB}{0,0,128}

\def\XXint#1#2#3{{\setbox0=\hbox{$#1{#2#3}{\int}$ }
		\vcenter{\hbox{$#2#3$ }}\kern-.57\wd0}}

\newcommand{\dx}{\,\mathrm{d}x}

\newcommand{\ds}{\,\mathrm{d}s}
\newcommand{\e}{\varepsilon}

\newcommand{\w}{\omega}

\newcommand{\R}{\mathbb{R}}
\newcommand{\Z}{\mathbb{Z}}
\newcommand{\N}{\mathbb{N}}
\newcommand{\Q}{\mathbb{Q}}
\newcommand{\Sph}{\mathbb{S}}
\newcommand{\Sd}{\Sph^{d-1}}
\newcommand{\M}{\mathcal{M}}
\renewcommand{\a}{\alpha}

\newcommand{\Hd}{\mathcal{H}^{d-1}}
\newcommand{\dHd}{\,\mathrm{d}\Hd}
\newcommand{\dH}{\,\mathrm{d}\mathcal{H}}

\renewcommand{\t}{\vartheta}
\newcommand{\E}{\mathbb{E}}
\newcommand{\A}{\mathcal{A}}
\newcommand{\Adm}{\mathscr{A}}
\newcommand{\F}{\mathcal{F}}
\newcommand{\I}{\mathcal{I}}
\renewcommand{\P}{\mathbb{P}}

\newcommand{\Eg}{E_g}

\newcommand{\Egw}{\Eg(\omega)}

\newcommand{\Hn}{H^\nu}
\newcommand{\Hnx}{\Hn(x_0)}
\newcommand{\Qn}{Q^\nu}
\newcommand{\Qnr}{\Qn_\rho}
\newcommand{\Qnxr}{\Qnr(x_0)}
\newcommand{\Br}{B_\rho}
\newcommand{\Brx}{\Br(x_0)}
\newcommand{\Rntl}{R_{t,\ell}^\nu}

\newcommand{\uzn}{u^{a,b,\nu}}

\newcommand{\eg}{\,{e.g.}, }
\newcommand{\ie}{\,{i.e.}, }

\newcommand{\defas}{:=}

\begin{document}
\author{Annika Bach}
\address[Annika Bach]{Dipartimento di Matematica ``Guido Castelnuovo'', Sapienza Universit\`{a} di Roma, Piazzale Aldo Moro 2,
00185 Roma, Italy}
\email{annika.bach@uniroma1.it}

\author{Matthias Ruf}
\address[Matthias Ruf]{Section de mathématiques, Ecole Polytechnique Fédérale de Lausanne, Station 8, 1015 Lausanne, Switzerland}
\email{matthias.ruf@epfl.ch}

\title[Fluctuation estimates for stochastic homogenization of partitions]{Fluctuation estimates for the multi-cell formula in stochastic homogenization of partitions}

\begin{abstract}
	In this paper we derive quantitative estimates in the context of stochastic homogenization for integral functionals defined on finite partitions, where the random surface integrand is assumed to be stationary. Requiring the integrand to satisfy in addition a multiscale functional inequality, we control quantitatively the fluctuations of the asymptotic cell formulas defining the homogenized surface integrand. As a byproduct we obtain a simplified cell formula where we replace cubes by almost flat hyperrectangles.
\end{abstract}

\maketitle
%{\small
%	\noindent\keywords{\textbf{Keywords:} Finite partitions, stochastic homogenization, fluctuation estimates}
	
%	\noindent\subjclass{\textbf{MSC 2020:} 49J45, 49Q20, 49J55, 60K37}
%}	

%\tableofcontents

\section{Introduction}
In a nutshell, stochastic homogenization deals with (mostly physical) problems in an environment that changes on a very small scale, but with a spatially heterogeneous distribution. These scales can enter for instance through oscillating $x$-dependent coefficients of a PDE or integrands of an integral functional. As the scale of random heterogeneous oscillations gets smaller and smaller or, equivalently, the surrounding space invades the whole space, one aims to derive an effective, averaged model in a spirit similar to the law of large numbers. In the context of linear elliptic PDEs of the form
\begin{equation}\label{intro:elliptic}
	-\text{div}(A(x/\e,\w)\nabla u)=f,
\end{equation}
where $\e$ denotes the scale of the fine oscillations and $\omega$ belongs to the sample space $\Omega$ of an underlying probability space $(\Omega,\mathcal F,\P)$, the first \emph{qualitative} results date back to the work of Kozlov \cite{Koz} and Papanicolaou and Varadhan \cite{PaVa}. 
In a variational context, first qualitative results in the nonlinear setting were obtained by Dal Maso and Modica~\cite{DMMoI,DMMoII}, where the authors derive an effective, averaged model for integral functionals of the form
\begin{equation}\label{intro:volume-integral}
	\int_{\Omega}L(x/\e,\w,\nabla u)\dx
\end{equation}
defined for \emph{Sobolev functions}.
%Having at hand....Based on these qualitative results there is an ever growing interest in deriving
Starting from these qualitative results, the interest grew in deriving error estimates for the homogenization approximation, which led to the development of a \emph{quantitative} theory for stochastic homogenization.
In the context of linear elliptic PDEs as in~\eqref{intro:elliptic}, first quantitative convergence results are already contained in \cite{Yur} and an unpublished preprint by Naddaf and Spencer \cite{NaSp}, the latter being optimal in the regime of a small ellipticity ratio. Then, in a non-perturbative regime there has been enormous progress in recent years starting with the works of Gloria and Otto \cite{GO_Var, GO_Error, GO_Jems} and Gloria, Neukamm and Otto \cite{GNO_Inv} (partially for discrete equations on a lattice). In terms of stochastic integrability of the error estimates, a breakthrough came with the work of Armstrong and Smart \cite{ArmSma}, which also covers the behavior of minimizers of functionals as in~\eqref{intro:volume-integral} under the assumption of uniform convexity, thus
giving the first quantitative version of the results in \cite{DMMoI,DMMoII}. Further quantitative results in this nonlinear setting were obtained for instance in \cite{AM,FiNe}. Finally, in the linear elliptic setting~\eqref{intro:elliptic} essentially optimal results in terms of scaling and stochastic integrability of weighted averages of the first order corrector were obtained in \cite{AKM,GO_Semi}. This list is by no means exhaustive as different assumptions on the statistics of the medium lead to different results. 

\smallskip

In the last decade, qualitative stochastic homogenization has been extended to variational models involving discontinuities, namely, to functionals defined on the space of \emph{functions of bounded variation}. In particular, models for interfacial energies were studied first in discrete environments where the underlying medium is either a periodic lattice with random interactions \cite{BP} or a random point set \cite{ACR,Bra-Cic-Ruf,BP20}. The latter results then motivated the qualitative analysis carried out in \cite{R17, BCR19} for random discrete approximations of free-discontinuity functionals (see also \cite{TrSl} for a discrete approximation of the total variation on point clouds).
In a continuum environment, the qualitative theory covers by now, among others, the stochastic homogenization of free-discontinuity functionals with randomly oscillating integrands \cite{CDMSZ19,CDMSZ21} or defined on randomly perforated domains~\cite{PSZ20}, of diffuse interface problems~\cite{Morfe}, and of singularly-perturbed elliptic approximations of free-discontinuity functionals~\cite{BMZ21}.

\smallskip

In this paper we derive first quantitative results for interfacial energies in a continuum medium. Namely, we consider energies acting on \emph{finite partitions}, \ie functions of bounded variation taking values in a finite set. More precisely, given an open set $D\subset\R^d$ with Lipschitz boundary and $\M\subset\R^m$ finite, we consider energies of the form
\begin{equation}\label{intro:def-energy}
E_{g,\e}(u,D)=\int_{D\cap S_u}g(\w,x/\e,u^+-u^-,\nu_u)\,\mathrm{d}\mathcal{H}^{d-1},\qquad u\in BV(D;\M),
\end{equation}
where $S_u$, $u^\pm$ and $\nu_u$ denote the jumpset of $u$, the traces of $u$ on both sides of $S_u$, and the generalized normal to $S_u$, respectively, while $g$ is a jointly measurable and uniformly bounded function (see Definition~\ref{defadmissible}).
%\begin{itemize}
%	\item $u\in BV(D;\mathcal{M})$ with $D\subset\R^d$ is an open set with Lipschitz boundary and $\mathcal{M}\subset\R^m$ is a finite set;
%	\item $S_u$ is the jump set of $u$ and $u^+$ and $u^-$ denote the traces on the two sides of $S_u$, while $\nu_u$ is the normal vector to $S_u$;
%	\item $g$ is jointly measurable and uniformly bounded from above and below (see Definition~\ref{defadmissible}.
%\end{itemize}
From a deterministic point of view, \ie when $g$ is independent of $\w$, functionals as in~\eqref{intro:def-energy} have been studied by Ambrosio and Braides in \cite{AmBrI,AmBrII} and we refer to those two papers for more details. In particular, in~\cite{AmBrII} the authors prove a periodic homogenization result for functionals as in~\eqref{intro:def-energy}. Although a corresponding stochastic homogenization result for stationary random integrands $g$ is, to the best of our knowledge, not available in the literature, it follows from by now standard methods. Indeed, the following result can be proved by following essentially the lines of~\cite[Theorem 3.2]{AmBrI} (using the integral representation result~\cite{BFLM}) and~\cite[Theorem 3.12]{CDMSZ19} (restricting to functions taking only finitely many values):
{\renewcommand{\thetheorem}{0.\arabic{theorem}}
\begin{theorem}
Let $g$ be an admissible random surface tension in the sense of Definition~\ref{defadmissible} satisfying Assumption \ref{ass:1} (see Section~\ref{s.preliminaries}). As $\e\searrow 0$ the sequence of functionals $E_{g,\e}$ defined in~\eqref{intro:def-energy} $\Gamma$-converges with respect to the strong $L^1(D)$-convergence to the functional
\begin{equation}\label{intro:Ehom}
E_{g_{\rm hom}}(u,D)=\int_{D\cap S_u}g_{\rm hom}(\w,u^+,u^-,\nu_u)\,\mathrm{d}\mathcal{H}^{d-1},
\end{equation}
where $g_{\rm hom}$ is as in~\eqref{eq:cellformula}. Moreover, if $g$ is ergodic, then $g_{\rm hom}$ is deterministic.
\end{theorem}
}
%
%From a probabilistic point of view, we further assume that $g$ is stationary in the spatial variable and satisfies a multiscale functional inequality of the form
%\begin{equation*}
%	{\rm Var}(X(g))\leq C\,\mathbb{E}\left[\int_0^{\infty}\int_{\R^d}\left(\partial^{\rm osc}_{g,B_{\ell+1}(x)}X(g)\right)^2\,\mathrm{d}x\,(\ell+1)^{-d}\pi(\ell)\,\mathrm{d}\ell\right],
%\end{equation*}
%for all measurable functions $X$, a non-negative weight $\pi$ and the oscillation $\partial^{\rm osc}_{g,B_{\ell+1}(x)}X(g)$ measures the sensitivity of $X(g)$ with respect to local perturbations of $g$ on $B_{\ell+1}(x)$ (see Assumption 2 for the details and Remark \ref{r.onMSG} for further comments). Using by-now standard methods one can prove that stationarity implies that the sequence of random functionals $E_{g,\e}$ $L^1(D)-\Gamma$-converges almost surely to an effective functional of the form
%\begin{equation*}
%E_{g_{\rm hom}}(u,D)=\int_{D\cap S_u}g_{\rm hom}(\w,u^+,u^-,\nu_u)\,\mathrm{d}\mathcal{H}^{d-1}(x),
%\end{equation*}
The effective integrand $g_{\rm hom}$ in~\eqref{intro:Ehom} does not depend on the spatial variable (but, depending on the set $\mathcal{M}$, it can loose the structural dependence on $u^+-u^-$). Moreover, it is given by an asymptotic minimization problem involving boundary conditions on larger and larger cubes. More precisely, denoting by $\uzn\in BV_{\rm loc}(\R^d,\mathcal{M})$ the function
\begin{equation*}
	\uzn(x)=\begin{cases}
		b &\mbox{if $\langle x,\nu\rangle>0$,}\\
		a &\mbox{otherwise,}
	\end{cases}
\end{equation*}
the integrand $g_{\rm hom}(\w,a,b,\nu)$ is given by the following limit, which in particular exists almost surely (see also Theorem~\ref{t.compressedprocess}):
\begin{equation}\label{eq:cellformula}
g_{\rm hom}(\w,a,b,\nu)=\lim_{t\to +\infty}\frac{1}{t^{d-1}}\inf\left\{E_{g,1}(u,tQ^{\nu}):\,u=\uzn\text{ in a neighborhood of }\partial tQ^{\nu}\right\},
\end{equation}
where $Q^{\nu}$ is a unit cube with two sides orthogonal to $\nu$. 

\smallskip
The aim of the present paper is to analyze the asymptotic formula in~\eqref{eq:cellformula} and give some quantitative error estimates when the size of the box grows. Clearly, in order to obtain quantitative information, mere stationarity will not suffice and we certainly have to focus on the ergodic setting, so that from now on we assume in this introduction that $g_{\rm hom}$ is deterministic. Denoting the value of the normalized infimum in \eqref{eq:cellformula} for fixed $t$ by $X_{t,t}^{a,b,\nu}(g)(\w)$ (the double index $t,t$ will become clear in a second), the error $X_{t,t}^{a,b,\nu}(g)(\w)-g_{\rm hom}(a,b,\nu)$ naturally splits into two parts, a deterministic error and and the fluctuations of $X_{t,t}^{a,b,\nu}(g)$. More precisely, we can write
\begin{equation}\label{intro:error}
X_{t,t}^{a,b,\nu}(g)(\w)-g_{\rm hom}(a,b,\nu)=\underbrace{X_{t,t}^{a,b,\nu}(g)(\w)-\mathbb{E}[X_{t,t}^{a,b,\nu}(g)]}_{\text{fluctuations}}+\underbrace{\mathbb{E}[X_{t,t}^{a,b,\nu}(g)]-g_{\rm hom}(a,b,\nu)}_{\text{deterministic error}}.
\end{equation}
In this paper we give quantitative estimates for the fluctuations assuming that the integrand $g$ in~\eqref{intro:def-energy} satisfies a multiscale functional inequality \cite{DG1,DG2} of the form
\begin{equation*}
	{\rm Var}(X(g))\leq C\,\mathbb{E}\left[\int_0^{\infty}\int_{\R^d}\left(\partial^{\rm osc}_{g,B_{s+1}(x)}X(g)\right)^2\,\mathrm{d}x\,(s+1)^{-d}\pi(s)\,\mathrm{d}s\right],
\end{equation*}
for all measurable functions $X$, a non-negative weight $\pi$ and the so-called oscillation $\partial^{\rm osc}_{g,B_{s+1}(x)}X(g)$. The latter measures the sensitivity of $X(g)$ with respect to local perturbations of $g$ on $B_{s+1}(x)$ (see Assumption \ref{ass:2} for the details and Remark \ref{r.onMSG} for further comments).
Under this additional assumption, one can show that
\begin{equation*}
\E\big[\big|X_{t,t}^{a,b,\nu}(g)-\E[X_{t,t}^{a,b,\nu}(g)]\big|^{2p}\big]\leq (C p^2)^p t^{p(2-d)}\int_0^{+\infty}(s+1)^{2p(d-1)}\pi(s)\ds\quad\text{for any}\ p\geq 1.
\end{equation*}
However, this estimate does not imply any decay in $t$ in dimension $2$. For this reason, we first move a step back and characterize $g_{\rm hom}$ in a way which is more convenient for our purpose. Namely, we prove that for any fixed $\nu\in\Sd$, in~\eqref{eq:cellformula} one can reduce the height of the cubes $tQ^{\nu}$ (\ie the side length along the direction $\nu$) to an arbitrarily slowly diverging sequence $\ell_t$. Denoting the corresponding value of the minimization problem by $X_{t,\ell_t}^{a,b,\nu}(g)(\w)$ (see Section~\ref{s.quantities} for a precise definition), we show that as long as $\ell_t\leq t$ diverges, it holds that
\begin{equation}\label{intro:planelike}
g_{\rm hom}(\w,a,b,\nu)=\lim_{t\to +\infty}X_{t,\ell_t}^{a,b,\nu}(g)(\w)	
\end{equation}
almost surely. Note that this result just requires $\R^d$-stationarity (the case of $\mathbb{Z}^d$-stationarity causes problems along irrational directions, see Section \ref{s.Zd} for how to circumvent this issue under additional assumptions). With the formula using the flat hyperrectangles (we refer to it also as almost planelike formula) we can improve the estimate on the fluctuations to
\begin{equation}\label{intro:fluctuations}
	\E\big[\big|X_{t,\ell_t}^{a,b,\nu}(g)-\E[X_{t,\ell_t}^{a,b,\nu}(g)]\big|^{2p}\big]\leq (C p^2)^p t^{p(1-d)}\ell_t^p\int_0^{+\infty}(s+1)^{2p(d-1)}\pi(s)\ds,
\end{equation}
which decays in any dimension $d\geq 2$ when $\ell_t$ grows much slower than $t$. Depending on the decay of the weight this bound can be summed with respect to $p$ and we obtain strong concentration estimates for the fluctuations. We present them in detail for an exponentially decaying weight in Corollary \ref{c.expweight}.

\smallskip
Characterizing $g_{\rm hom}$ via the almost planelike formula in~\eqref{intro:planelike} is of interest on its own, in particular in comparison with the periodic setting in the case of two phases.
Namely, when the function $g$ is deterministic and $1$-periodic in the spatial variable and $\M$ contains only two values, Caffarelli and de la Llave have shown in \cite{CadlL} under a mild continuity and ellipticity assumption that there always exist so-called planelike minimizers, that means, the jump set stays in a uniformly bounded neighborhood of the hyperplane orthogonal to $\nu$.
As a consequence, in this setting the (deterministic) limit in~\eqref{intro:planelike} still holds when the height $\ell_t$ does not diverge, but is uniformly bounded.
However, in the random setting such a property is not expected to hold. Indeed, in Section~\ref{s.nonplanelike} we construct an example of a stationary, ergodic integrand $g$ such that for any fixed $\ell\in\N$ the corresponding limit of $X_{t,\ell}^{a,b,\nu}(g)(\w)$ as $t\to +\infty$ is strictly smaller than $g_{\rm hom}$. Moreover, for first passage percolation in dimension two (which is equivalent to a lattice-based version of the problem defining $X_{t,t}^{a,b,\nu}(g)(\w)$ with two phases) it is expected that the maximal deviation from the straight line connecting $0$ and $n\nu^{\perp}$ is of the order $n^{2/3}$ (see \cite[Section 4.2]{50years} and references therein). Note that our result, which extends to discrete models without significant changes, is no contradiction to this conjecture, since we only speak about the minimal energy value instead of absolute minimizers. 

\smallskip
We close this introduction by briefly commenting on our result together with our choice of methods, its limitations and possible future problems.
The strong concentration estimates for the fluctuations obtained in the present paper allow to control the probabilistic error in~\eqref{intro:error}. Instead, a quantitative estimate for the deterministic error in~\eqref{intro:error} is beyond the scope of this paper. In fact, such a control seems to be a rather difficult issue for subadditive processes and one of the few general methods seems to be the theory developed in \cite{Alex}. However, the assumptions therein are not well-adaptable to random interfaces except in dimension two where the duality to paths can be used.
Let us also mention that in the case of two phases, there is a similar problem to the minimization problem defining $X_{t,\ell_t}^{a,b,\nu}$, which consists of finding the maximal flow/minimal cut between the upper and lower parts of the boundary of the cube with iid weighted edges given by nearest neighbors in the integer lattice $\mathbb{Z}^d$ (the problems are slightly different since the minimal cut does not have to be the discontinuity set of a function). Under much weaker moment conditions on the weights than uniform boundedness from above and below, fluctuation estimates on the energy of a minimal cut were obtained for instance in \cite{RoThe,The}. While parts of the analysis seems to be adaptable to our continuum model assuming a finite range of dependence, some estimates use the independence assumption through a logarithmic Sobolev inequality relying on \cite{BoLuMa}. For a continuum model, requiring a logarithmic Sobolev inequality and finite range of dependence seems quite restrictive. Moreover, as shown in \cite{DG2}, many physical models satisfy a weighted functional inequality since they are transformations of product structures.
Finally, note that in contrast to uniformly convex problems, quantitative results on the energy of minimizers do not imply any quantitative estimates on the convergence rate of the minimizers of interfacial-type energies. The latter seems to be a very challenging problem for the future.

\section{Preliminaries and notation}\label{s.preliminaries}
\subsection{General notation}
We first introduce some notation that will be used in this paper. Given a measurable set $A\subset\R^d$ we denote by $|A|$ its $d$-dimensional Lebesgue measure, and by $\mathcal{H}^{k}(A)$ its $k$-dimensional Hausdorff measure. For $x\in\R^d$ we denote by $|x|$ the Euclidean norm and $\Brx$ denotes the open ball with radius $\rho>0$ centered at $x_0\in\R^d$. If $x_0=0$ we simply write $\Br$.
Given $x_0\in\R^d$ and $\nu\in\Sph^{d-1}$ we let $\Hnx$ be the hyperplane orthogonal to $\nu$ and passing through $x_0$ and for every $(a,b)\in\R^m\times\R^m$, the piecewise constant function taking values $a,b$ and jumping across $\Hnx$ is denoted by $u_{x_0}^{a,b,\nu}:\R^{d}\to\R^{m}$,\ie
\begin{equation}\label{eq:purejump}
u_{x_0}^{a,b,\nu}(x):=\begin{cases} b &\mbox{if $\langle x-x_0,\nu\rangle >0$,}
\\
a &\mbox{otherwise,}
\end{cases}
\end{equation}
where the brackets $\langle\cdot,\cdot\rangle$ denote the standard scalar product. If $x_0=0$ we write $\Hn$ and $\uzn$ in place of $H^\nu(0)$ and $u_{0}^{a,b,\nu}$, respectively.
Let $\{e_1,\ldots,e_d\}$ be the standard basis of $\R^d$. Then $O_{\nu}$ is the orthogonal matrix induced by the linear mapping
\begin{equation}\label{eq:matrix}
x\mapsto \begin{cases}
\displaystyle{2\frac{\langle x,\nu+e_d\rangle}{|\nu+e_d|^2}(\nu+e_d)-x} &\mbox{if $\nu\in \mathbb{S}^{d-1}\setminus\{-e_d\}$,}
\\
-x &\mbox{otherwise.}
\end{cases}
\end{equation}
In this way, $O_\nu e_d=\nu$ and the set $\{\nu_j\defas O_\nu e_j\colon j=1,\ldots, d-1\}$ is an orthonormal basis for $\Hn$. 
Setting $\nu_d=\nu$, we define the cube $\Qn$ as
\begin{equation}\label{eq:defcube}
\Qn=\left\{x\in\R^d\colon |\langle x,\nu_j\rangle| <1/2\text{ for }j=1,\ldots,d\right\},
\end{equation}
and we set $\Qnxr=x_0+\rho \Qn$. 
%We further denote by $\Hnx$ the hyperplane orthogonal to $\nu$ and passing through $x_0$. If $x_0=0$ we write $\Hn$. 
%
%
%
%For $x_0\in \R^d$, $\nu\in \mathbb{S}^{d-1}$ and $\zeta\in\R^m$ define the function $\uxzn:\R^{d}\to\R^{m}$ as
%\begin{equation}\label{eq:purejump}
%\uxzn(x):=\begin{cases} \zeta &\mbox{if $\langle x-x_0,\nu\rangle >0$,}
%\\
%0 &\mbox{otherwise.}
%\end{cases}
%\end{equation}
%If $x_0=0$ we simply write $\uzn$ in place of $u_{0}^{\zeta,\nu}$.

Finally, the letter $C$ stands for a generic positive constant that may change every time it appears.
\subsection{BV-functions}

The relevant function space in this paper is the space of \emph{finite partitions},\ie the space of \emph{functions of bounded variation} taking only finitely many values. More precisely, we denote by $BV(D;\R^m)$ the space of all functions $u\in L^1(D;\R^m)$ whose distributional derivative $Du$ is a matrix-valued Radon measure. Moreover, given $\M\subset\R^m$, we set $BV(D;\M)\defas\{u\in BV(D;\R^m)\colon u(x)\in\M\text{ a.e.\! in }D\}$. If $\M$ is finite, then $Du$ can be represented as $Du(B)=\int_{B\cap S_u}(u^+(x)-u^-(x))\otimes\nu_u(x)\dHd$ for any Borel set $B\subset D$. Here $S_u$ is the so-called jumpset of $u$, which is $\Hd$-rectifiable and coincides $\Hd$-a.e. with the complement in $D$ of Lebesgue points of $u$. Moreover, $\nu_u(x)$ is the measure-theoretic normal to $S_u$ and $u^+(x),u^-(x)$ are the traces on both sides of $S_u$. We refer the reader to~\cite{AFP} for more details on functions of bounded variation.
\subsection{Boundedness and probabilistic assumptions}
In this subsection we give the precise assumptions we make on the random integrand. We start fixing some notation in the deterministic setting. Namely, for a given parameter $c\geq 1$ we denote by $\A_c$ the class of all Borel measurable functions $g:\R^d\times \mathbb\R^m\times \mathbb{S}^{d-1}\to [0,+\infty)$ satisfying
%there exists $c>0$ such that 
\begin{equation}\label{cond:Ac}
\frac{1}{c}\leq g(x,\zeta,\nu)\leq c,\quad \text{and}\quad g(x,\zeta,\nu)=g(x,-\zeta,-\nu)\,,  
\end{equation}
for every $(x,\zeta,\nu)\in\R^d\times \mathbb\R^m\times \mathbb{S}^{d-1}$.
To any $g\in\A_c$ and $D\subset\R^d$ open with Lipschitz boundary we associate a functional $\Eg(\cdot,D)$ defined on partitions by setting $\Eg(\cdot,D):L^1_{\rm loc}(\R^d;\R^m)\to[0,+\infty]$,
\begin{equation}\label{def:energy-deterministic}
\Eg(u,D)\defas
\begin{cases}
\displaystyle \int_{D\cap S_u} g(x,u^+-u^-,\nu_u)\dHd &\text{if }u\in BV(D;\M),\\
+\infty &\text{otherwise in }L^1_{\rm loc}(\R^d;\R^m).
\end{cases}
\end{equation}
Here $\mathcal{M}\subset\R^m$ is a finite set that we fix throughout this paper. Note that $\Eg$ is well-defined for $g\in\A_c$ thanks to the second condition in~\eqref{cond:Ac} and the fact that the triple $(u^+(x),u^-(x),\nu_u(x))$ is uniquely defined up to a permutation in $(u^+(x),u^-(x))$ and a simultaneous change of sign in $\nu_u(x)$.

We are now in a position to rigorously introduce the random setting.
Throughout the paper $(\Omega,\F,\P)$ denotes a complete probability space.
\begin{definition}[Admissible surface tensions]\label{defadmissible} 
We say that a function $g:\Omega\times\R^d\times\R^m\times\mathbb{S}^{d-1}\to [0,+\infty)$ is an admissible random surface tension, if it is jointly measurable and there exists $\geq 1$
such that for every $\omega\in\Omega$ the function $g(\omega):\R^d\times\R^m\times\Sph^{d-1}\to[0,+\infty)$, $g(\omega)\defas g(\omega,\cdot,\cdot,\cdot)$ belongs to $\mathcal{A}_c$.
\end{definition}
For any admissible random surface tension $g$ and for $\omega\in\Omega$ we set $\Egw\defas E_{g(\w)}$, where $E_{g(\w)}$ is defined according to~\eqref{def:energy-deterministic}, \ie $\Egw(u,D)=\int_{D\cap S_u}g(\w,x,u^+-u^-,\nu_u)\dHd$ for $D\subset\R^d$ open, $u\in BV(D;\M)$.

\medskip

We now introduce two further probabilistic assumptions. The first one concerns spatial stationarity of the integrand, while the second one is a multi-scale functional inequality (or weighted spectral gap). We start by recalling the notion of measure-preserving group actions.
\begin{definition}[Measure-preserving group action]\label{def:group-action} Let $k\in \N$, $k\geq 1$. A \emph{measure-preserving additive group action} on $(\Omega,\F,\P)$  is a family $\{\tau_z\}_{z\in\R^k}$ of mappings $\tau_z:\Omega\to\Omega$ satisfying the following properties:
\begin{enumerate}[label=(\arabic*)]
	\item\label{meas} (measurability) $\tau_z$ is $\F$-measurable for every $z\in\R^k$;
	\item\label{inv} (invariance) $\P(\tau_z A)=\P(A)$, for every $A\in\F$ and every $z\in\R^k$;
	\item\label{group} (group property) $\tau_0=\rm id_\Omega$ and $\tau_{z_1+z_2}=\tau_{z_2}\circ\tau_{z_1}$ for every $z_1,z_2\in\R^k$.
\end{enumerate}
If, in addition, $\{\tau_z\}_{z\in\R^k}$ satisfies the implication
\begin{equation*}
\mathbb{P}(\tau_zA\Delta A)=0\quad\forall\, z\in\R^k\implies \mathbb{P}(A)\in\{0,1\},
\end{equation*}
then it is called ergodic.
\end{definition}
\begin{remark}[Discrete measure-preserving group action]\label{r.group-action}
If in Definition~\ref{def:group-action} the space $\R^k$ is replaced by $\Z^k$, we say that the corresponding family $\{\tau_z\}_{z\in\Z^k}$ of mappings $\tau_z:\Omega\to\Omega$ satisfying~\ref{meas}--\ref{group} is a \emph{discrete} measure-preserving group action.
\end{remark}
We are now in a position to state our probabilistic assumptions on the random surface tension $g$.
\begin{assumption}\label{ass:1}
The admissible random surface tension is $\R^d$-stationary,\ie there exists a measure-preserving group action $\{\tau_{z}\}_{z\in\R^d}$ such that for all $\omega\in\Omega$ and for all $z\in\R^d$ it holds that
\begin{equation*}
g(\tau_z\w,x,\zeta,\nu)=g(\w,x+z,\zeta,\nu)\quad\quad\forall\, (x,\zeta,\nu)\in \R^d\times\R^m\times\mathbb{S}^{d-1}.
\end{equation*}
It is called ergodic, if it is stationary and the group action $\{\tau_z\}_{z\in\R^d}$ is ergodic. We refer to Assumption~1(E) if ergodicity holds.
\end{assumption}
\begin{assumption}\label{ass:2}
Let $\pi\in L^1((0,+\infty))$ be non-negative. The admissible random surface tension $g$ satisfies a multiscale functional inequality with weight $\pi$ and with respect to the oscillation,\ie for any function $X:\mathcal{A}_c\to\R$ such that $\w\mapsto (X(g))(\omega)\defas X(g(\w))$ is measurable we have 
\begin{equation*}
{\rm Var}(X(g))\leq C\,\mathbb{E}\left[\int_0^{\infty}\int_{\R^d}\left(\partial^{\rm osc}_{g,B_{s+1}(x)}X(g)\right)^2\,\mathrm{d}x\,(s+1)^{-d}\pi(s)\,\mathrm{d}s\right],
\end{equation*}
where the oscillation of $X$ with respect to $g$ on $U\subset\R^d$ is formally\footnote{As already noted in \cite{DG2} this definition is not measurable in general, so that one should define it either using the conditional essential supremum \cite{BaCaRe} or as the measurable envelope of the above definition. However, in this paper we will only use measurable bounds on the oscillation, so that these issues do not matter.} defined as
\begin{equation}\label{def:oscillation}
\begin{split}
\partial^{\rm osc}_{g,U}X(g)(\w):=&{\rm ess}\,\sup \{X(g')\colon g'\in\mathcal{A}_c,\,g'|_{(\R^d\setminus U)\times \R^m\times\mathbb{S}^{d-1}}=g(\w)|_{(\R^d\setminus U)\times \R^m\times\mathbb{S}^{d-1}}\}
\\
&-{\rm ess}\,\inf \{X(g')\colon g'\in\mathcal{A}_c,\,g'|_{(\R^d\setminus U)\times \R^m\times\mathbb{S}^{d-1}}=g(\w)|_{(\R^d\setminus U)\times \R^m\times\mathbb{S}^{d-1}}\}.
\end{split}
\end{equation}
\end{assumption} 
\begin{remark}\label{r.onMSG}
Our definition of multiscale functional inequality differs from \cite{DG1,DG2}, since we consider functions $g$ not just depending on $x$. However, to have some concrete examples we can consider surface tensions of the form $g(\w,x,\zeta,\nu)=a(\w,x)\phi(\zeta,\nu)$ with $a$ satisfying a multiscale functional inequality in the spirit of \cite{DG2}. We chose our framework to allow for more general dependencies that are not present in the homogenization of linear elliptic PDEs. 
\end{remark}
\subsection{Relevant quantities}\label{s.quantities}
We now introduce the quantities which are relevant for the analysis carried out in the present paper.

For $\nu\in \mathbb{S}^{d-1}$ let $O_\nu$ be the orthogonal matrix introduced in~\eqref{eq:matrix}, $\nu_j\defas O_\nu e_j$, $j=1,\ldots, d-1$, and for every $t,\ell>0$ set
\begin{equation}\label{eq:defrectangle}
\Rntl\defas \{x\in\R^d\colon |\langle x,\nu\rangle|<\ell/2,\ |\langle x,\nu_j\rangle|<t/2, j=1,\ldots,d-1\}.
\end{equation}
Let $g$ be an admissible random surface tension; for $t,\ell>0$, $(a,b)\in\M\times\M$, $\nu\in \mathbb{S}^{d-1}$ and for every $\w\in\Omega$ we introduce the quantity 
\begin{equation}\label{def:Xtl}
X_{t,\ell}^{a,b,\nu}(g)(\omega):=t^{1-d}\inf\big\{\Egw (u,\Rntl)\colon u\in\Adm(\uzn,\Rntl)\big\},
\end{equation}
where for every $D\subset\R^d$ open with Lipschitz boundary and $\bar{u}\in BV(D;\M)$ we set
\begin{equation}\label{def:adm-class}
\Adm(\bar u,D)\defas\{u\in BV(D;\M),\ u=\bar{u}\text{ in a neighborhood of }\partial D\}.
\end{equation}
\begin{remark}\label{r.Xtl}
Clearly, for $\ell=t$ we have $R_{t,t}^\nu=t\Qn$, so that in particular
\begin{equation}\label{def:Xtt}
X_{t,t}^{a,b,\nu}(g)(\omega)=t^{1-d}\inf\big\{\Egw (u,t\Qn)\colon u\in \Adm(\uzn,tQ^\nu)\big\}.
\end{equation}
Moreover, it is immediate to see that for every $(a,b)\in\M\times\M$, $t>0$ and $\omega\in\Omega$ the mapping $\ell\mapsto X_{t,\ell}^{a,b,\nu}(g)(\omega)$ is decreasing in $\ell$. In fact, if $\ell'\geq\ell>0$, then any competitor $u\in\Adm(\uzn,\Rntl)$ can be extended to a competitor $u\in\Adm(\uzn,R_{t,\ell'}^\nu)$ by setting $u\defas\uzn$ in $R_{t,\ell'}^\nu\setminus\Rntl$. Since $S_{\uzn}\cap(R_{t,\ell'}^\nu\setminus\Rntl)=\emptyset$, by definition of $\Egw$ we have $\Egw(u,R_{t,\ell'}^\nu)=\Egw(u,R_{t,\ell}^\nu)$ and we conclude by minimization that 
\begin{equation}\label{est:monotonicity-ell}
X_{t,\ell'}^{a,b,\nu}(g)(\omega)\leq X_{t,\ell}^{a,b,\nu}(g)(\omega)\quad\text{for every}\ \ell'\geq\ell>0.
\end{equation}
Using the same extension argument for fixed $\ell$, but different parameters $t,t'$, shows that the mapping $t\mapsto X_{t,\ell}^{a,b,\nu}(g)(\omega)$ is almost decreasing. Namely, using~\eqref{cond:Ac} we obtain
\begin{equation}\label{est:monotonicity-t}
X_{t',\ell}^{a,b,\nu}(g)(\omega)\leq \bigg(\frac{t}{t'}\bigg)^{d-1}X_{t,\ell}^{a,b,\nu}(g)(\omega)+c\frac{(t'-t)(t')^{d-2}}{(t')^{d-1}}\leq X_{t,\ell}^{a,b,\nu}(g)(\omega)+\frac{c(t'-t)}{t'}\quad\text{for every}\ t'\geq t.
\end{equation}
Finally, note that by testing the function $u=\uzn$ in the infimum problem, we deduce from the boundedness of $g$ that
\begin{equation}\label{eq:energybound}
	0\leq  X_{t,\ell}^{a,b,\nu}(g)(\w)\leq c
\end{equation}
uniformly in all parameters. Eventually, thanks to~\cite[Proposition A.1]{CDMSZ19}, for any admissible random surface tenstion $g$, the mapping $\w\mapsto X_{t,\ell}^{a,b,\nu}(g)(\w)$ is measurable. Thus, we can define the oscillation of $X_{t,\ell}^{a,b,\nu}$ according to~\eqref{def:oscillation} and use~\eqref{eq:energybound} with $g'$ in place of $g$ to obtain the immediate bound
\begin{equation}\label{est:oscillation-trivial}
\partial_{g,B_{s+1}(x)}^{\rm osc}X_{t,\ell}^{a,b,\nu}(g)(\w)\leq 2c,
\end{equation}
for any $s>0$ and $x\in\R^d$.
\end{remark}
\section{Statement of the main results}
In this section we present our main results. The corresponding proofs are postponed to Section \ref{s.proofs}. Our first result provides a simpler formula to compute the asymptotic surface tension $g_{\rm hom}$. More precisely, we show that instead of the full cube $tQ^{\nu}$, we can reduce the size in the direction $\nu$ taking an arbitrary slow diverging sequence $\ell_t$ instead of $t$. This result is also interesting from a numerical point of view.
\begin{theorem}\label{t.compressedprocess}
Let $g:\Omega\times\R^d\times\R^m\times\Sph^{d-1}\to[0,+\infty)$ be an admissible random surface tension satisfying Assumption \ref{ass:1}. Then there exists a function $g_{\rm hom}:\Omega\times\R^m\times\R^m\times\Sph^{d-1}\to[0,+\infty)$ such that almost surely for all $(a,b)\in\M\times\M$, $\nu\in\Sph^{d-1}$ we have
\begin{equation}\label{ex:limit:cubes}
\lim_{t\to+\infty} X_{t,t}^{a,b,\nu}(g)(\omega)=g_{\rm hom}(\omega,a,b,\nu).
\end{equation}
In particular, the limit in~\eqref{ex:limit:cubes} extists. Moreover, $g_{\rm hom}$ is $\{\tau_z\}_{z\in\R^d}$ invariant. If, in addition, $\{\tau_z\}_{z\in\R^d}$ is ergodic, then $g_{\rm hom}$ is deterministic.
%, \ie $\tau_z(\widehat{\Omega})=\widehat{\Omega}$ and $g_{\rm hom}(\tau_z\omega,a,b,\nu)=g_{\rm hom}(\omega,a,b,\nu)$ for every $\omega\in\widehat{\Omega}$ and $z\in\R^d$.
Finally, let $\nu\in\Sph^{d-1}$ be fixed; then almost surely, for every $(a,b)\in\M\times\M$ we have
\begin{equation}\label{ex:limit:compressed}
\lim_{t\to +\infty}X_{t,\ell_t}^{a,b,\nu}(g)(\omega)=g_{\rm hom}(\omega,a,b,\nu),
\end{equation}
under the assumption that $0<\ell_t\leq t$ satisfies $\displaystyle\lim_{t\to +\infty}\ell_t=+\infty$.	
\end{theorem}
\begin{remark}\label{r.mainresult}
	\begin{enumerate}[leftmargin=2em, label=(\roman*)]
		\item\label{rem:rational} The exceptional set where the convergence in \eqref{ex:limit:compressed} might fail may depend on $\nu$, but not on the sequence $\ell_t$. 
In the particular case $\ell_t=t$ it is possible to find a set of full probability where the convergence in~\eqref{ex:limit:cubes} holds for all directions $\nu$. Indeed, in this case it is sufficient to establish~\eqref{ex:limit:cubes} on a countable dense subset of $\Sd$ and extend the convergence to all directions via a deterministic continuity argument (see,\eg \cite[Lemma 5.5]{CDMSZ19}).
Choosing the flat hyperrectangles rules out this possibility. 
This poses additional problems when the medium satisfies only discrete stationarity (as discrete environments on a lattice), since in this case we are only able to prove~\eqref{ex:limit:compressed} for rational directions. However, under Assumption \ref{ass:2} we obtain a slightly weaker version of~\eqref{ex:limit:compressed} also in the case of $\mathbb{Z}^d$-stationarity (see Section \ref{s.Zd} and Corollary~\ref{c.asirrational}). The limit in \eqref{ex:limit:cubes} holds without any extra assumption for $\Z^d$-stationary models thanks to the above mentioned continuity argument.
		\item Due to \eqref{eq:energybound} the almost sure convergence implies convergence in $L^p(\Omega)$ for any $1\leq p<+\infty$. Denoting by $\mathcal{F}_{\rm inv}$ the $\sigma$-algebra of $\{\tau_z\}_{z\in\R^d}$-invariant sets, then by subadditivity and stationarity (cf. Lemma~\ref{l.ellfixed} and its proof) the conditional expectations satisfy $\mathbb{E}[X_{t,t}^{a,b,\nu}|\mathcal{F}_{\rm inv}](\w)\geq g_{\rm hom}(\w,a,b,\nu)$. Moreover, thanks to the monotinicity Property~\eqref{est:monotonicity-ell} we have
		\[\mathbb{E}[X_{t,\ell_t}^{a,b,\nu}|\mathcal{F}_{\rm inv}](\w)\geq\mathbb{E}[X_{t,t}^{a,b,\nu}|\mathcal{F}_{\rm inv}](\w)\geq g_{\rm hom}(\w,a,b,\nu).\]
		In the ergodic case the above estimate reduces to the expectation. 
		In this sense, the formula with the flat hyperrectangles produces a larger deterministic error, but allows at the same time to obtain concentration estimates for the fluctuations (cf. Corollary~\ref{c.expweight}).  
	\end{enumerate}
\end{remark}
Our next result gives a control of the variance (and higher moments) under the additional Assumption~\ref{ass:2}. Note that in dimension $2$ the flatness of the hyperrectangles is crucial to obtain a decay rate.
\begin{theorem}\label{t.concentration}
	Let $g$ be an admissible random surface tension satisfying Assumption \ref{ass:1}\footnote{Strictly speaking, Assumption \ref{ass:1} is not needed in the proof. However, usually Assumption \ref{ass:2} is established for at least $\Z^d$-stationary media.}  and \ref{ass:2}. Then there exists a constant $c_d>0$ such that for all $p\geq 1$ and every $(a,b)\in\M\times\M$, $\nu\in\Sd$ the estimate
	\begin{equation}\label{est:variance}
		\E\big[\big(X_{t,\ell_t}^{a,b,\nu}(g)-\E[X_{t,\ell_t}^{a,b,\nu}(g)]\big)^{2p}\big]\leq (c_d p^2)^p t^{p(1-d)}\ell_t^p\int_0^{+\infty}(s+1)^{2p(d-1)}\pi(s)\ds
	\end{equation}
	holds whenever $t\geq\ell_t\geq 1$. In particular,
	\begin{equation*}
		{\rm Var}\left(X_{t,\ell_t}^{a,b,\nu}(g)\right)\leq c_d t^{1-d}\ell_t\int_0^{+\infty}(s+1)^{2(d-1)}\pi(s)\ds.
	\end{equation*}
\end{theorem}
We deduce asymptotic exponential concentration estimates for the fluctuations in case of an exponential decay of the weight $\pi$. This is the case for many physical models of random heterogeneities (cf. \cite[Section 3]{DG2}).
\begin{corollary}\label{c.expweight}
	Let $g$ be an admissible random surface tension satisfying Assumptions \ref{ass:1}(E) and \ref{ass:2}. Assume that the weight $\pi$ satisfies $\pi(s)\leq C\exp(-s/C)$ for some $C>0$. Then there exists $C_d>0$ such that for all $t\geq \ell_t\geq 1$ and $(a,b)\in\M\times\M$, $\nu\in\Sd$ we have
	\begin{equation*}
		\mathbb{E}\left[\exp\left(\frac{1}{C_d}\left|\frac{X_{t,\ell_t}^{a,b,\nu}(g)-\mathbb{E}[X_{t,\ell_t}^{a,b,\nu}(g)]}{\sqrt{t^{1-d}\ell_t}}\right|^{\frac{1}{d}}\right)\right]\leq 4.
	\end{equation*}
	In particular, for every $\eta>0$ we have
	\begin{equation*}
		\limsup_{t\to +\infty}\bigg((t^{1-d}\ell_t)^{\tfrac{1}{2d}} \log\Big(\mathbb{P}\big(|X_{t,\ell_t}^{a,b,\nu}(g)-g_{\rm hom}(a,b,\nu)|>\eta\big)\Big)\bigg)<0.
	\end{equation*} 
	
\end{corollary}
\begin{remark}[On the (non)-optimality of the concentration estimates]\label{r.optimal}
In the proof of Theorem \ref{t.concentration} we estimate the oscillation of the process $X_{t,\ell_t}^{a,b,\nu}(g)$ for all balls $B_{s+1}(x)$ that intersect the hyperrectangle $R_{t,\ell_t}^{\nu}$, which then leads by integration to the factor $\ell_t$. It would suffice to consider all balls that intersect the jumpset of a minimizer for $X_{t,\ell_t}^{a,b,\nu}$.
However, there are two problems: in general, minimizers do not exist. Moreover, choosing an almost minimizer, one has then to estimate the measure of the $2(s+1)$-neighborhood of the jumpset for all $s>0$. If one assumes that the weight $\pi$ has compact support, then one could try to use the theory of Minkowski content. However, this needs to be done in a quantitative way since (almost) minimizers depend on $t$ and $s>0$ is not infinitesimal, but finite. For the moment this seems to be out of reach for non-smooth $x$-dependent integrands $g$. We remark that in a discrete setting, this approach seems more plausible since there one just has to estimate the number of edges used in a minimal interface (which is proportional to $t^{d-1}$ when the weights are uniformly bounded from above and below). Finally, even with the best possible assumption that the jumpset of an almost minimizer is flat, the improvement would be minor since the factor $\ell_t$ can diverge arbitrarily slow. Eventually, the exponent $\frac{1}{2d}$ in Corollary \ref{c.expweight} is due to two facts: the factor $2$ can be avoided if we use a functional derivative instead of the oscillation together with a logarithmic Sobolev inequality instead of the variance control via the spectral gap (see also \cite[Proposition 1.10 i)]{DG1}). The factor $d$ disappears when we assume that $\pi$ is bounded and has compact support. 
\end{remark}
\section{Further remarks}
In this section we discuss how our strategy has to be adapted if one assumes only stationarity with respect to integer translations in $\Z^d$ (cf. Remark~\ref{r.group-action}). This is particularly relevant for discrete interface models, which are often defined on the lattice $\Z^d$, where integer translations provide a natural framework. Moreover, we give an example of a stationary, ergodic integrand which shows that, in general, the assumption $\ell_t\to+\infty$ in Theorem \ref{t.compressedprocess} cannot be dropped. 
\subsection{On $\Z^d$-stationary integrands}\label{s.Zd} 
Assuming that stationarity of the random surface tension $g$ is only satisfied for a discrete measure preserving group action $\{\tau_z\}_{z\in\Z^d}$, our strategy of proof establishes the limit~\eqref{ex:limit:compressed} in Theorem \ref{t.compressedprocess} only for rational directions,\ie for $\nu\in \mathbb{S}^{d-1}\cap\Q^{d}$.
In contrast to the process $X_{t,t}^{a,b,\nu}(g)$, the hyperrectangles in the definition of $X_{t,\ell_t}^{a,b,\nu}$ do not allow to use deterministic continuity arguments to extend the convergence to irrational directions (cf. Remark~\ref{r.mainresult}~\ref{rem:rational}), since a small, but fixed rotation does not lead to a uniformly small perturbation of the thin hyperrectangles when $t$ grows. However, under Assumption \ref{ass:2} with a certain decay of the weight $\pi$ we can extend the convergence to all directions using Theorem~\ref{t.concentration}. Since Assumption \ref{ass:2} only allows to control the variance of random variables, we first need to prove the convergence of the expectation of the random variables $X_{t,\ell_t}^{a,b,\nu}(g)$. This will be achieved again under the sole assumption of stationarity. A similar approach has been used in \cite[Proposition 3.5]{RoThe} for a maximal flow model on $\Z^d$.
\begin{proposition}\label{p.expectation}
	Let $g:\Omega\times\R^d\times\R^m\times\Sph^{d-1}\to[0,+\infty)$ be an admissible random surface tension satisfying Assumption \ref{ass:1} with a measure preserving group action $\{\tau_z\}_{z\in\Z^d}$. Then for every $(a,b)\in\M\times\M$ and every $\nu\in\mathbb{S}^{d-1}$ it holds that
	\begin{equation*}
		\lim_{t\to +\infty}\mathbb{E}[X_{t,\ell_t}^{a,b,\nu}(g)]=\mathbb{E}[g_{\rm hom}(\cdot,a,b,\nu)]
	\end{equation*}
	under the assumption that $0<\ell_t\leq t$ satisfies $\lim_{t\to +\infty}\ell_t=+\infty$. If $g$ is ergodic, then $\mathbb{E}[g_{\rm hom}(\cdot,a,b,\nu)]=g_{\rm hom}(a,b,\nu)$ since $g_{\rm hom}$ is deterministic.
\end{proposition}
%\begin{remark}
%	It will be clear from the proof that the result still holds if we take the conditional expectation with respect to the $\sigma$-algebra of $\{\tau_z\}_{z\in\Z^d}$-invariant sets under which $g_{\rm hom}$ is invariant (see Theorem~\ref{t.limit:cubes}). In this case the limit of the conditional expectations equals $g_{\rm hom}(\omega,a,b,\nu)$.
%\end{remark}
\begin{remark}
	Since the condition $\ell_t\leq t$ implies that $X_{t,\ell_t}^{a,b,\nu}(g)\geq X_{t,t}^{a,b,\nu}(g)$ (cf. Remark \ref{r.Xtl}), the convergence of the expectations to the same limit implies also convergence in $L^1(\Omega)$ and therefore also almost sure convergence after selecting a subsequence. From this one can deduce the convergence in $L^p(\Omega)$ for any $1\leq p<+\infty$ by the dominated convergence theorem.
\end{remark}
Combining the above result with the concentration estimate in Theorem \ref{t.concentration}, we obtain the almost sure convergence along all directions also for $\Z^d$-stationary models.
\begin{corollary}\label{c.asirrational}
	Let $g$ be an admissible random surface tension satisfying Assumption \ref{ass:1}(E) with a measure preserving group action $\{\tau_z\}_{z\in\Z^d}$\footnote{In \cite{DG1} it was shown that the integrability of $\pi$ already implies the ergodicity, so this assumption is rendundant. Even though the proof was given under the assumption of continuum stationarity, the proof remains unchanged with $\mathbb{Z}^d$-stationarity.} and Assumption \ref{ass:2} with a weight $\pi$ that satisfies
	\begin{equation*}
		\int_{0}^{+\infty}(s+1)^{r}\pi(s)\,\mathrm{d}s<+\infty\quad\text{for some}\ r>2(d-1).
	\end{equation*}
	Fix $\nu\in\mathbb{S}^{d-1}$ and let $\bar{\ell}_n\to+\infty$ be an arbitrary diverging sequence. Then a.s. for all $(a,b)\in\mathcal{M}\times\M$ it holds that
	\begin{equation}\label{e:ghom-irrational}
		\lim_{n\to +\infty}X_{t_n,\ell_n}^{a,b,\nu}(g)(\omega)=g_{\rm hom}(a,b,\nu),
	\end{equation}
	whenever $\bar{\ell}_n\leq\ell_n\leq t_n$ for every $n\in\N$.
\end{corollary}
\begin{remark}[Choice of $\bar{\ell}_n$]
%Without choosing a sequence $\bar{\ell}_n$ in Corollary~\ref{c.asirrational}, the exceptional set where~\eqref{e:ghom-irrational} might fail may depend on the diverging sequences $\ell_n,t_n$. 
The choice of the sequence $\bar{\ell}_n$ in Corollary~\ref{c.asirrational} is only for technical reasons to avoid that the exceptional set where~\eqref{e:ghom-irrational} might fail may depend on the diverging sequences $\ell_n,t_n$. Instead, choosing an arbitrarily slowly diverging sequence $\bar{\ell}_n$ allows to exclude an exceptional set depending only on $\nu$ and the fixed sequence $\bar{\ell}_n$. As already observed in Remark~\ref{r.mainresult}~\ref{rem:rational}, this procedure is not necessary for rational directions.
\end{remark}
\subsection{Non-existence of plane-like minimizing sequences in the ergodic setting}\label{s.nonplanelike}
In light of Theorem \ref{t.compressedprocess} one might try to prove that the height of the flat hyperrectangles can be taken to be bounded, still obtaining the same limit. Below we show that this is in general not possible. To be more precise, let us introduce a notion of plane-like minimizing sequences. In what follows we focus on the ergodic setting to simplify the formulas.
\begin{definition}[Plane-like minimizing sequence]\label{def:planelike}
	Let $g$ be an admissible random surface tension satisfying Assumption \ref{ass:1}(E) and let $g_{\rm hom}$ be as in Theorem~\ref{t.compressedprocess} and $\nu\in\Sd$. Given a realization $\omega\in\Omega$ we say that a sequence $(u_t^\omega)_t$ is a minimizing sequence for $g_{\rm hom}(a,b,\nu)$, if $u_t^\omega\in\Adm(\uzn,tQ^\nu)$ for every $t>0$ and 
	\begin{equation}\label{cond:min-sequence}
		g_{\rm hom}(a,b,\nu)=\lim_{t\to+\infty} t^{1-d}E_g(\omega)(u_t^\omega,tQ^\nu).
	\end{equation}
	We say that a minimizing sequence for $g_{\rm hom}(a,b,\nu)$ is  plane-like, if there exist $\ell\in\N$ and $t_\ell>0$ such that $u_t^\omega\in\Adm(\uzn,R_{t,\ell}^\nu)$ for all $t>t_\ell$.
\end{definition}
We will provide an example of a stationary, ergodic integrand $g$ in dimension two such that the probability of finding a plane-like minimizing sequence for the direction $\nu=(0,1)$ is zero.
\begin{example}\label{ex.ex}
Let $d=2$ and $\mathcal{M}=\{0,1\}$. There exists an admissible random surface tension $g$ satisfying Assumption \ref{ass:1}(E) such that with probability $1$ there exists no plane-like minimizing sequence for $g_{\rm hom}(0,1,e_2)$.
\end{example}
\begin{remark}
The restriction to $\mathcal{M}=\{0,1\}$ is only for convenience. In fact, we will construct an integrand depending only on $x$ and $\omega$, so that the same construction works for any finite set $\mathcal{M}$. Moreover, it will be clear from the construction that it can be extended to any dimension upon heavier notation and replacing $e_2$ by $e_d$.
\end{remark}
Below we prove the claim made in Example \ref{ex.ex} and construct an admissible ergodic random surface tension $g$ for which $g_{\rm hom}(0,1,e_2)$ has no plane-like minimizing sequences in the sense of Definition~\ref{def:planelike}. We let $u^{0,1}$ be as in~\eqref{eq:purejump} with $(a,b)=(0,1)$ and $\nu=e_2$.

\textbf{Step 1.} In this step we construct a suitable random surface integrand $g$. To this end, we let $\{X_i\}_{i\in\mathbb{Z}}$ be a sequence of independent and $[1,2]$-uniformly distributed random variables on a suitable probability space $(\Omega,\mathcal{F},\mathbb{P})$ equipped with a measure-preserving, ergodic map $\tau:\Omega\rightarrow\Omega$ satisfying the following properties:
\begin{enumerate}[label=\roman*)]
	\item\label{inverse} $\tau$ is bijective and the inverse map $\tau^{-1}:\Omega\to\Omega$ is again $\F$-measurable;
	\item\label{iteration} $X_i(\w)=X_0(\tau^i\omega)$ for every $i\in\Z$, where $\tau^i$ denotes the $i$-times iterated composition of the map $\tau$ for $i\geq 0$ (with the convention $\tau^0\defas{\rm id}$), respectively the $-i$-times iterated composition of $\tau^{-1}$ for $i<0$.
\end{enumerate} 
This setting can be realized on the product space $\Omega=[1,2]^{\mathbb{Z}}$ with the shift operator (see~\cite[Section 7.3]{JKO_Hom}).
We now define a random surface integrand $g:\Omega\times\R^2\times\Sph^1\to[0,+\infty)$ by setting
\begin{equation}\label{constr:random-integrand}
	g(\omega,x,\nu)\defas X_{\lceil x_2\rceil}(\omega)\quad\text{for every}\ x=(x_1,x_2)\in\R^2,\ \omega\in\Omega,\ \nu\in\Sph^{1},
\end{equation}
where $\lceil x_2\rceil$ denotes the upper integer part of $x_2$. In this way, $g$ is measurable, and since each $X_i$ takes values in $[1,2]$ it satisfies $1\leq g(\omega,x,\nu)\leq 2$, \ie $g(\omega)\in\A_2$. Moreover, both $\tau$ and $\tau^{-1}$ are measure-preserving and ergodic. Thus, the family of maps $\tau_z\defas\tau^{z_2}$, $z\in\Z^2$ defines a measure-preserving ergodic group action. Thanks to~\ref{iteration} it is immediate to see that $g$ is stationary with respect to $\{\tau_z\}_{z\in\Z^2}$ as above\footnote{This integrand is only $\Z^2$-stationary. However, following the method in \cite[Section 7.3]{JKO_Hom} one can turn this example in an $\R^d$-stationary, ergodic medium that has pointwise the same stripe-like structure.}, and we set
\begin{equation}\label{def:energy-example}
	E_g(\omega)\defas\int_{S_u\cap D}g(\omega,x,\nu)\dH^1=\int_{S_u\cap D}X_{\lceil x_2\rceil}(\omega)\dH^1\quad\text{for all}\ u\in BV(D;\{0,1\}).
\end{equation}
Eventually, for every $\ell\in\N$ and $t>0$ we let $X_{t,\ell}^{e_2}$ be as in~\eqref{def:Xtl} with $\nu=e_2$, $a=0$, $b=1$, and $E_g$ as in~\eqref{def:energy-example}. 
%Recall that thanks to~\eqref{ex:limit:allell},~\eqref{claim:limitell} and~\eqref{mono:ell} we have that almost surely 
%%
%\begin{equation}\label{ex:limit-example}
%g_\ell(\omega,e_2)=\lim_{t\to+\infty} X_{t,2\ell}^{e_2}(g)(\omega),\quad g_{\rm hom}(e_2)=\lim_{\ell\to+\infty} g_\ell(\omega,e_2),\quad\text{and}\quad g_{\rm hom}(e_2)\leq g_\ell(\omega,e_2),
%\end{equation}
%%
%for every $\ell\in\N$.

\noindent Let us introduce the random variables $Y_\ell:\Omega\to[1,2]$ given by $Y_\ell(\omega)\defas\min_{i\in [-\ell+1,\ell]} X_i(\omega)$, which clearly satisfy $Y_{\ell+1}(\omega)\leq Y_\ell(\omega)$ for every $\ell\in\N$. Morever, for every $\ell\in\N$ there exists $\Omega_\ell\in\F$ with $\P(\Omega_\ell)>0$ such that
\begin{equation}\label{cond:minXi}
	Y_\ell(\omega)> X_{\ell+1}=Y_{\ell+1}(\omega)\quad\text{for every}\ \omega\in\Omega_\ell.
\end{equation} 
Indeed, since all $X_i$ are independent and uniformly distributed on the interval $[1,2]$, for any $s\in(1,2)$ we have
$
\P\big(X_{-\ell}>s,\ldots,X_{\ell}>s,X_{\ell+1}\leq s\big)=(2-s)^{2\ell+1}(s-1)>0,
$
which implies~\eqref{cond:minXi}.

\medskip

\textbf{Step 2.} In this step we show that almost surely we have
\begin{enumerate}[label=\arabic*)]
	%\item\label{prop:ghom-example} $g_{\rm hom}(e_2)=1$ almost surely;
	%\item\label{prop:gell-example} for every $\ell\in\N$ there exists $\Omega_\ell\in\F$ with $\P(\Omega_\ell)>0$ and $g_{\ell}(\omega,e_2)>g_{\ell+1}(\omega,e_2)$ for every $\omega\in\Omega_\ell$.
	\item\label{prop:Yell}$\lim_{t\to +\infty}X_{t,2\ell}^{e_2}(g)(\omega)=Y_\ell(\omega)$ for every $\ell\in\N$;
	\item\label{prop:lim-ell}$\lim_{\ell\to +\infty}Y_\ell(\omega)=1$.
\end{enumerate}
%
%Clearly,~\ref{prop:gell-example} together with the inequality in~\eqref{ex:limit-example} implies that for every $\ell\in\N$ we have $g_\ell(\omega,e_2)>g_{\rm hom}(e_2)$ with positive probability.
%To establish~\ref{prop:ghom-example} and~\ref{prop:gell-example} it is convenient to prove that for any $\ell\in\N$ we have
%%
%\begin{equation}\label{claim:gell-example}
%g_\ell(\omega,e_2)=\min_{i\in [-\ell+1,\ell]} X_i(\omega)=:Y_\ell(\omega).
%\end{equation}
%
We start proving~\ref{prop:Yell}: Since $-\ell+1\leq\lceil x_2\rceil\leq\ell$ for every $x\in R_{t,2\ell}^{e_2}=(-\frac{t}{2},\frac{t}{2})\times(-\ell,\ell)$, we clearly have that
$
E_g(\omega)(u, R_{t,2\ell}^{e_2})\geq Y_\ell(\omega)\Hd(S_u\cap R_{t,2\ell}^{e_2})
$
for every $u\in BV(R_{t,2\ell}^{e_2};\{0,1\})$,
which in particular yields
\begin{equation}\label{est:X2ell}
	X_{t,2\ell}^{e_2}(g)(\omega)\geq Y_\ell(\omega)\frac{1}{t}\min\big\{\mathcal{H}^1(S_u\cap R_{t,2\ell}^{e_2})\colon u\in\Adm (u^{0,1}, R_{t,2\ell}^{e_2})\big\}=Y_{\ell}(\w)\frac{1}{t}\mathcal{H}^1 (H^{\nu}\cap R_{t,2\ell}^{e_2})=Y_\ell(\omega).
\end{equation}
To estimate $X_{t,2\ell}^{e_2}(g)(\omega)$ from above, let $i_\ell\in(-\ell,\ell]$ (depending also on $\omega$) be such that $X_{i_\ell}(\omega)\leq X_i(\omega)$ for every $i\in(-\ell,\ell]$. Without loss of generality we assume that $i_\ell>0$. Then the function $u_\ell$ defined as the characteristic function of the set $R_{t,2\ell}^{e_2}\setminus R_{t-1,2i_\ell-1}^{e_2}\setminus\{x_2<0\}$ (see Figure~\ref{fig:uell}) is admissible for $X_{t,2\ell}^{e_2}(g)(\omega)$ and satisfies
\begin{equation*}
	\begin{split}
		E_g(\omega)(u_\ell,R_{t,2\ell}^{e_2}) &\leq\int_{\{x_2=i_\ell-\frac{1}{2},\ |x_1|<\frac{t-1}{2}\}}X_{\lceil x_2\rceil}(\omega)\dH^1\\
		&\hspace*{1em}+2\mathcal{H}^1\Big(\big\{|x_1|=\tfrac{t-1}{2},\ x_2\in\big(0,i_\ell-\tfrac{1}{2}\big)\big\}\cup\big\{|x_2|=0,\ |x_1|\in\big(\tfrac{t-1}{2},\tfrac{t}{2}\big)\big\}\Big)\\
		&= (t-1)X_{i_\ell}(\omega)+4i_\ell\leq t Y_\ell(\omega)+4\ell.
	\end{split}
\end{equation*}
\begin{figure}[t]
	\centering
	\def\svgwidth{.45\textwidth}
	\input{uell.pdf_tex}
	\caption{The rectangles $R_{t,2\ell}^{e_2}$, $R_{t-1,2i_\ell-1}^{e_2}$ and in gray the set $R_{t,2\ell}^{e_2}\setminus R_{t-1,2i_\ell-1}^{e_2}\setminus \{x_2<0\}$ where $u_\ell=1$.}
	\label{fig:uell}
\end{figure}
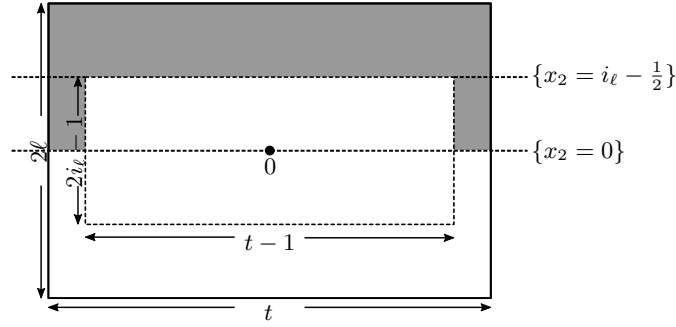
Dividing the above inequality by $t$ and using~\eqref{est:X2ell} we thus obtain
\begin{equation*}
	X_{t,2\ell}^{e_2}(g)(\omega)\leq\frac{1}{t}E_g(u_\ell,R_{t,2\ell}^{e_2})\leq Y_\ell(\omega)+\frac{4\ell}{t}\leq X_{t,2\ell}^{e_2}(g)(\omega)+\frac{4\ell}{t}.
\end{equation*}
Passing in the above inequality to the limsup on the left-hand side and to the liminf on the right-hand side yields that almost surely there exists
\begin{equation}\label{example:gell}
	g_{\ell}(\omega,e_2)\defas\lim_{t\to+\infty}X_{t,2\ell}^{e_2}(g)(\omega)=Y_{\ell}(\omega).
\end{equation}
%\medskip
%We finally observe that for fixed $\ell\in\N$ there exists a set $\Omega_\ell\in\F$ with $\P(\Omega_\ell)>0$ such that 
%
%\begin{equation}\label{cond:minXi}
%Y_\ell(\omega)> X_{\ell +1}(\w)= Y_{\ell+1}(\omega)\quad\text{for every}\ \omega\in\Omega_\ell.
%\end{equation} 
%%
%Indeed, since all $X_i$ are independent and uniformly distributed on the interval $[1,2]$, for any $s\in(1,2)$ we have
%$
%\P\big(X_{-\ell}>s,\ldots,X_{\ell}>s,X_{\ell+1}\leq s\big)=(2-s)^{2\ell+1}(s-1)>0,
%$
%%
%which ensures~\eqref{cond:minXi}. Thus, Property~\ref{prop:gell-example} follows from~\eqref{claim:gell-example}.
We now come to prove~\ref{prop:lim-ell}; by construction we have $\P(Y_\ell>s)=(2-s)^{2\ell}$ for every $s\in(1,2)$, hence $\sum_{\ell}\P(Y_\ell>s)<+\infty$, since $2-s<1$. As a consequence, the monotone continuity of $\P$ together with the Borel-Cantelli Lemma imply that
\begin{equation*}
	\P\big(\limsup_{\ell\to+\infty} Y_\ell>1\big)=\lim_{s\searrow 1}\P\big(\limsup_{\ell\to+\infty} Y_\ell>s\big)=0.
\end{equation*}
Thus,~\ref{prop:lim-ell} follows immediately from the existence of $\lim_\ell Y_\ell(\omega)=\inf_\ell Y_\ell(\omega)\geq 1$.

\medskip

\textbf{Step 3.} Conclusion and final remarks.\\
\noindent As a first consequence of~\ref{prop:Yell} and~\ref{prop:lim-ell} we deduce that almost surely there exists
\begin{equation}\label{example:ghom}
	g_{\rm hom}(e_2)\defas\lim_{t\to+\infty}X_{t,t}^{e_2}(g)(\omega)=1\leq g_{\ell}(\omega,e_2),
\end{equation}
where $g_\ell$ is as in~\eqref{example:gell}. Indeed, arguing as in~\eqref{est:X2ell} we obtain $1\leq\liminf_t X_{t,t}^{e_2}(g)(\omega)$, while~\eqref{est:monotonicity-ell} together with~\ref{prop:Yell} implies that $\limsup_t X_{t,t}^{e_2}(g)(\omega)\leq Y_\ell(\omega)$ for every $\ell\in\N$. Hence,~\eqref{example:ghom} follows by letting $\ell\to+\infty$ and using~\ref{prop:lim-ell}.

As a consequence, with probability $1$ minimizing sequences are not plane-like. In fact, assume that for $\w\in\Omega$ there exists a plane-like minimizing sequence $(u_t^\omega)_t$ in the sense of Definition~\ref{def:planelike} with parameter $2\ell=2\ell(\w)\in\N$. Then by definition
\begin{equation*}
	1=g_{\rm hom}(e_2)=\lim_{t\to +\infty}\frac{1}{t}E_g(\w)(u_t^{\w},tQ^{\nu})=\lim_{t\to +\infty}\frac{1}{t}E_g(\w)(u_t^{\w},R_{t,2\ell}^{\nu})\geq \lim_{t\to +\infty}X_{t,2\ell}^{e_2}(g)(\w)=Y_{\ell}(\w)\geq 1,
\end{equation*}
where we used \eqref{example:gell}. Hence $Y_{\ell}(\w)=1$, but we clearly have $\mathbb{P}(\exists \ell\in\N:\,Y_{\ell}=1)=0$.
%We also notice that~\eqref{cond:minXi} together with~\eqref{example:gell} and~\eqref{example:ghom} implies that for every $\ell\in\N$ we have $g_\ell(\omega,e_2)>g_{\rm hom}(e_2)$ with positive probability. As a consequence, minimizing sequences for $g_{\rm hom}(e_2)$ are not plane-like. This can be seen as follows: For any $\omega\in\Omega$ let $(u_t^\omega)_t$ be a minimizing sequence in the sense of Definition~\ref{def:planelike}. Moreover, for any $\w\in\Omega_\ell$ let $\alpha_\ell^\w\defas 1-Y_\ell(\omega)>0$. In view of~\eqref{cond:min-sequence} and~\eqref{example:ghom}, for every $\w\in\Omega_\ell$ there exists $t_\ell^\omega>2\ell$ such that for every $t>t_\ell^\omega$ we have $\frac{1}{t}E_g(\omega)(u_t^\omega,tQ)\leq 1+\frac{\alpha_\ell^\w}{2}$. Assume now that $u_t^\omega\in\Adm(u^{0,1},R_{t,\ell}^{e_2})$ for some $t>t_\ell^\w$; then the positivity of the energy together with~\eqref{est:X2ell} implies that
%%
%\begin{align*}
%	1+\frac{\alpha_\ell^\w}{2}\geq \frac{1}{t}E_g(\omega)(u_t^\omega,tQ)\geq\frac{1}{t}E_g(\omega)(u_t^w,R_{t,\ell}^{e_2})\geq X_{t,2\ell}^{e_2}(g)(\omega)\geq Y_\ell(\omega)=1+\alpha_\ell^\w,
%\end{align*}
%%
%which leads to a contradiction. Hence, $u_t^\w$ cannot be plane-like in the sense of Definition~\ref{def:planelike}.

We conclude this example by observing that for every $\ell\in\N$ the function $g_\ell(\cdot,e_2)$ is neither deterministic nor invariant under the group action $\{\tau_z\}_{z\in\Z^2}$. In fact, for any interval $(s_1,s_2)\subset(1,2)$ we have by definition that $\P(Y_\ell\in(s_1,s_2))=(2-s_1)^{2\ell}-(2-s_2)^{2\ell}>0$, which implies that $g_\ell(\cdot,e_2)$ still depends on the realization $\omega$. This already implies that $g_\ell(\cdot, e_2)$ cannot be invariant under the group action $\{\tau_z\}_{z\in\Z^2}$. However, this can also be see directly by observing that by construction, for any $z=(z_1,z_2)\in\Z^2$ it holds that $g_\ell(\tau_z\omega,e_2)=g_{\ell+z_2}(\w,e_2)$. Thus, assuming that $z_2>0$, from~\eqref{example:gell} and~\eqref{cond:minXi} we deduce that
\begin{align*}
	g_\ell(\tau_z\omega,e_2)=g_{\ell+z_2}(\omega,e_2)\leq g_{\ell+1}(\omega,e_2)<g_\ell(\omega,e_2)\quad\text{for every}\ \omega\in\Omega_\ell.
\end{align*}
\qed
\section{Proofs}\label{s.proofs}
%For $\nu\in \mathbb{S}^{d-1}$ let $O_\nu$ be the orthogonal matrix introduce in~\eqref{eq:matrix}, $\nu_j\defas O_\nu e_j$, $j=1,\ldots, d-1$, and for every $t,\ell>0$ set
%%
%\begin{equation}\label{eq:defrectangle}
%\Rntl\defas \{x\in\R^d\colon |\langle x,\nu\rangle|<\ell/2,\ |\langle x,\nu_j\rangle|<t/2, j=1,\ldots,d-1\}.
%\end{equation}
%%
%Let $g$ be an admissible random surface tension; for $0<\ell\leq t$, $\zeta\in\R^m$, $\nu\in \mathbb{S}^{d-1}$ and for every $\w\in\Omega$ we introduce the quantity 
%%
%\begin{equation}\label{def:Xtl}
%X_{t,\ell}^{\zeta,\nu}(g)(\omega):=t^{1-d}\inf\Big\{\Egw (u,\Rntl)\colon u\in BV(\Rntl;\M),\ u=\uzn\text{ near }\partial\Rntl\Big\}.
%\end{equation}
%
\subsection{Almost plane-like formulas in the stationary setting: Proof of Theorem~\ref{t.compressedprocess}}
As a preliminary step towards the proof of Theorem~\ref{t.compressedprocess}, for $\ell\in\N$ fixed we analyze the asymptotic behavior of $X_{t,\ell}^{a,b,\nu}(g)(\w)$ as $t\to+\infty$. This will be done by relating $X_{t,\ell}^{a,b,\nu}(g)$ in a suitable way to a subadditive stochastic process in dimension $(d-1)$. We recall here the notion of such a process for the readers' convenience.

\medskip
For every $p=(p_1,\dots,p_{d-1}),\,q=(q_1,\dots,q_{d-1})\in\R^{d-1}$ with $p_i<q_i$ for all $i\in\{1,\ldots,d-1\}$ we consider the $(d-1)$-dimensional half-opens intervals 
\begin{equation*}
[p,q):=\{x\in\R^{d-1}\colon p_i\le x_i<q_i\,\,\text{for}\,\,i=1,\ldots,d-1\}
\end{equation*}
and we set
\begin{equation*}
\I:=\{[p,q)\colon p,q\in\R^{d-1}\,,\, p_i<q_i\,\,\text{for}\,\,i=1,\ldots,d-1\}\,.
\end{equation*}
\begin{definition}[Subadditive process]\label{sub_proc}
	A subadditive process with respect to a measure-preserving additive group action $\{\tau_z\}_{z\in\R^{d-1}}$ is a function $\mu\colon\I\times\Omega\to\R$ satisfying the following properties:
	\begin{enumerate}[label=(\arabic*)]
		\item\label{subad:meas} (measurability) for every $I\in\I$ the function $\omega\mapsto\mu(I,\omega)$ is $\F$-measurable;
		\item\label{subad:cov} (stationarity) for every $\omega\in\Omega$, $I\in\I$, and $z\in\R^{d-1}$ we have $\mu(I+z,\omega)=\mu(I,\tau_z(\omega))$;
		\item\label{subad:sub} (subadditivity) for every $I\in\I$ and for every finite partition $(I^i)_{i=1}^k$ of $I$, we have
		\begin{equation*}
		\mu(I,\omega)\leq\sum_{i=1}^k\mu(I_i,\omega)\quad\text{for every}\,\,\omega\in\Omega\,;
		\end{equation*}
		\item\label{subad:bound} (boundedness) there exists $M>0$ such that $0\leq\mu(I,\omega)\le M\mathcal{L}^{d-1}(I)$ for every $\omega\in\Omega$ and $I\in\I$.
	\end{enumerate}
\end{definition} 
In order to, relate $X_{t,\ell}^{a,b,\nu}(g)$ to a subadditive process, we associate to each hyperrectangle $\Rntl$ a set $I\in\mathcal{I}$ (and vice versa) as follows: For fixed $\nu\in\Sph^{d-1}$ we let $O_\nu$ be the orthogonal matrix induced by~\eqref{eq:matrix}. 
%Since $\nu\in\mathbb{S}^{d-1}\cap\mathbb{Q}^d$ is a rational direction it follows that $O_\nu\in\mathbb{Q}^{d\times d}$, so that there exists an integer $m_\nu$ such that $m_\nu O_\nu(z,0)\in\Z^d$ for every $z\in\Z^{d-1}$. 
For every $I=[p_1,q_1)\times\cdots\times[p_{d-1},q_{d-1})\in\mathcal{I}$ we denote by $s_{\rm max}(I)\defas\max_{i}|q_i-p_i|$ its maximal side length and define the open set $Q^{\ell}(I)\subset\R^d$ as
\begin{equation}\label{def:ndimintervals}
Q^{\ell,\nu}(I):=O_\nu\Big({\rm int}\, I\times \min\{\ell, s_{\rm max}(I)\}(-1/2,1/2)\Big),
\end{equation}
where ${\rm int}$ denotes the $(d-1)$-dimensional interior. For $\ell=+\infty$ we clearly have $Q^{\infty,\nu}(I)=O_\nu\big({\rm int}\, I\times s_{\rm max}(I)(-1/2,1/2)\big)$.
Then we define a function $\mu_{\ell}^{a,b,\nu}:\mathcal{I}\times\Omega\to\R$ by setting
\begin{equation}\label{def:subadprocess}
\mu_{\ell}^{a,b,\nu}(I,\omega)\defas\inf\{\Egw (u,Q^{\ell,\nu}(I))\colon u\in\Adm(u^{a,b,\nu},Q^{\ell}(I))\}.
\end{equation}
In this way, we have
\begin{equation}\label{relation:mu-ell-X-t-ell}
\frac{\mu_{\ell}^{a,b,\nu}\big([-\tfrac{t}{2},\tfrac{t}{2})^{d-1},\omega\big)}{t^{d-1}}=X_{t,\ell}^{a,b,\nu}(g)(\omega)\;\text{ if }\;\ell\leq t,\qquad \frac{\mu_{\ell}^{a,b,\nu}\big([-\tfrac{t}{2},\tfrac{t}{2})^{d-1},\omega\big)}{t^{d-1}}=X_{t,t}^{a,b,\nu}(g)(\omega)\;\text{ if }\;\ell\geq t.
\end{equation}
\begin{lemma}\label{l.ellfixed}
Let $g:\Omega\times\R^d\times\R^m\times\Sph^{d-1}\to[0,+\infty)$ be an admissible random surface tension satisfying Assumption \ref{ass:1}. For every $\ell\in\N\cup\{+\infty\}$, $(a,b)\in\M\times\M$, and $\nu\in\Sph^{d-1}$ let $\mu_{\ell}^{a,b,\nu}(I,\omega)$ be as in~\eqref{def:subadprocess}. Then there exists a measure-preserving group action $\{\tau_z^\nu\}_{z\in\R^{d-1}}$ such that $\mu_{\ell}^{a,b,\nu}$ is a subadditive process with respect to $\{\tau_{z}^\nu\}_{z\in\R^{d-1}}$ satisfying
\begin{equation}\label{eq:l_inftybound}
\mu_{\ell}^{a,b,\nu}(I,\omega)\leq \Egw (\uzn,Q^\ell(I))\leq c  \mathcal{L}^{d-1}(I)
\end{equation}
for every $\omega\in\Omega$.
Moreover, there exist $\Omega_{\ell}^{a,b,\nu}\subset\Omega$ with $\P(\Omega_\ell^{a,b,\nu})=1$ and a function $g_\ell(\cdot,a,b,\nu):\Omega\to\R$ such that for every $\omega\in\Omega_{\ell}^{a,b,\nu}$ we have
\begin{align}\label{ex:limit:rational}
g_{\ell}(\w,a,b,\nu)=
\begin{cases}
\displaystyle\lim_{t\to+\infty}X_{t,\ell}^{a,b,\nu}(g)(\w)&\text{if}\ \ell\in\N,\cr
\displaystyle\lim_{t\to+\infty}X_{t,t}^{a,b,\nu}(g)(\omega) &\text{if}\ \ell=+\infty.
\end{cases}
\end{align}
Eventually, $\Omega_\ell^{a,b,\nu}$ and $g_\ell(\cdot,a,b,\nu)$ are invariant under the group action $\{\tau_z^\nu\}_{z\in\R^{d-1}}$, \ie $\tau_z^\nu(\Omega_{\ell}^{a,b,\nu})=\Omega_\ell^{a,b,\nu}$ for every $z\in\R^{d-1}$ and
\begin{equation}\label{eq:partialinvariance}
g_{\ell}(\tau_z^\nu\w,a,b,\nu)=g_{\ell}(\omega,a,b,\nu)\quad\text{for every}\ \omega\in\Omega_{\ell}^{a,b,\nu}.
\end{equation} 
\end{lemma}
\begin{remark}\label{r.ex:limit:rational}
For $\ell=+\infty$ the corresponding function $g_\infty(\cdot,a,b,\nu)$ given by~\eqref{ex:limit:rational} is invariant under the whole group action $\{\tau_z\}_{z\in\R^d}$ associated to $g$ (the invariance can be proven by a deterministic argument as in \cite[Theorem 6.2]{CDMSZ19} similar to~\eqref{est:extension}). This is, in general, not true for the functions $g_\ell(\cdot,a,b,\nu)$ with $\ell\in\N$. In fact, Example~\ref{ex.ex} provides an admissible random surface tension $g$ where~\eqref{eq:partialinvariance} fails if $\tau_z^\nu$ is replaced by $\tau_z$. As a consequence, $g_\infty(\cdot,a,b,\nu)$ is deterministic if $g$ is ergodic, while this is in general not true for $g_\ell$. 

\smallskip
If $g$ is only $\Z^d$-stationary, and $\nu\in\Sd\cap\mathbb Q^{d\times d}$, then $\mu_{\ell}^{a,b,\nu}$ defines a subadditive process with respect to a discrete measure-preserving group action. In particular, the limits in~\eqref{ex:limit:rational} still exist for rational directions. The existence of the second limit can still be extended to irrational directions via continuity (cf. Remark~\ref{r.mainresult}~\ref{rem:rational}).

\smallskip 
The sets of full probability for which~\eqref{ex:limit:rational} holds depend on $\ell,a,b,\nu$; for fixed $\nu\in\Sph^{d-1}$ we can define a set $\Omega^\nu$ of full probability by taking the countable intersection of $\Omega_\ell^{a,b,\nu}$ over $(a,b)\in\M\times\M$, $\ell\in\N\cup\{+\infty\}$. Then for every $\omega\in\Omega^\nu$ the limits in~\eqref{ex:limit:rational} exist for all $(a,b)\in\M\times\M$ and every $\ell\in\N\cup\{+\infty\}$. In particular, the monotonicity property in Remark implies that for every $\omega\in\Omega^\nu$ and every $(a,b)\in\M\times\M$ we have
\begin{equation}\label{monotonicity-gell}
g_\infty(\omega,a,b,\nu)\leq g_{\ell+1}(\omega,a,b,\nu)\leq g_\ell(\omega,a,b,\nu)\quad\text{for every}\ \ell\in\N.
\end{equation}
Example~\ref{ex.ex} also shows that for every $\ell\in\N$ the above inequality can be strict on a set of positive probability (depending again on $\ell,a,b,\nu$).
\end{remark}
\begin{proof}[Proof of Lemma \ref{l.ellfixed}]
Throughout this proof we fix $(a,b)\in\M\times\M$ and $\nu\in \mathbb{S}^{d-1}$. In order to not to overburden notation, for every $I\in\I$ and $\ell\in\N\cup\{+\infty\}$ we write $Q^\ell(I)=Q^{\ell,\nu}(I)$ for the $d$-dimensional interval introduced in~\eqref{def:ndimintervals} and for every $\omega\in\Omega$ we set $\mu_\ell(I,\omega)\defas\mu_\ell^{a,b,\nu}(I,\omega)$ with $\mu_\ell^{a,b,\nu}$ as in~\eqref{def:subadprocess}.
%We first define a suitable stationary, subadditive stochastic process for $\ell_t=\ell$ fixed and then we argue that we can switch the limits as $t\to +\infty$ and $\ell\to +\infty$ at least in conditional expectation. In what follows we let $\ell\in\N$ be fixed. Denote by $\mathcal{I}=\{[a,b):a,b\in\mathbb{Z}^{d-1}:a_i<b_i\text{ for } i=1,\ldots, d-1\}$ the class of half-open cubes with integer vertices in dimension $d-1$. A multi-parameter stochastic process is a map $\mu:\mathcal{I}\to L^1(\Omega)$. In order to take into account the quantity $X_{t,\ell}^{\zeta,\nu}(g)(\omega)$ we have to associate to each cuboid $\Rntl$ a set $I\in\mathcal{I}$ (and vice versa). 
%
%Since $\nu\in\mathbb{S}^{d-1}\cap\mathbb{Q}^d$ is a rational direction it follows that $O_\nu\in\mathbb{Q}^{d\times d}$, so that there exists an integer $m=m(\nu)$ such that $mO_\nu(z,0)\in\Z^d$ for every $z\in\Z^{d-1}$. For every $I=[a_1,b_1)\times\cdots\times[a_{d-1},b_{d-1})\in\mathcal{I}$ we denote by $s_{\rm max}(I)\defas\max_{i}|b_i-a_i|$ its maximal side length and define the open set $Q^{\ell}(I)\subset\R^d$ as
%%
%\begin{equation*}
%Q^{\ell}(I):=mO_\nu\Big({\rm int}\, I\times \min\{\ell, s_{\rm max}(I)\}(-1/2,1/2)\Big),
%\end{equation*}
%where ${\rm int}$ denotes the $(d-1)$-dimensional interior. Then define the stochastic process $\mu_{\ell}:\mathcal{I}\to L^1(\Omega)$ by
%\begin{equation*}
%\mu_{\ell}(I,\omega)\defas\inf\{\Egw (u,Q^{\ell}(I))\colon u\in BV(D;\M),\ u=\uzn\text{ near }\partial Q^\ell(I)\}.
%\end{equation*}
%%
%We divide the analysis of the process $\mu_{\ell}$ into several steps. 

\medskip
\textbf{Step 1.} Stationarity and subadditivity of $\mu_{\ell}$
\\
The fact that $\mu_{\ell}(I,\cdot)$ is measurable follows from~\cite[Proposition A.1]{CDMSZ19}. Moreover, since $g(\omega)\in\A_c$ for every $\omega\in\Omega$, we obtain the uniform bound~\eqref{eq:l_inftybound} by taking $\uzn$ as a candidate in the minimization problem defining $\mu_\ell(I,\omega)$ as in Remark~\ref{r.Xtl}. 

We next prove stationarity of the process. To this end, given $z\in\R^{d-1}$ we set $z^\nu:=O_\nu(z,0)\in\R^d$ and define a measure-preserving group action $\{\tau_z^\nu\}_{z\in\R^{d-1}}$ by setting $\tau_z^\nu:=\tau_{-z^\nu}$, where $\{\tau_z\}_{z\in\R^d}$ is the group action associated to $g$. Note that for every $I\in\mathcal{I}$ and every $z\in\R^{d-1}$ we have $Q^{\ell}(I-z)=Q^{\ell}(I)-z^\nu$. Thus, for any $u\in BV(Q^\ell(I-z);\M)$ the function $u_z\defas u(\cdot-z^\nu)$ belongs to $BV(Q^\ell(I);\M)$.
Moreover, $z^{\nu}\in \Hn$ due to the properties of $O_{\nu}$, which implies that $\uzn(\cdot-z^\nu)=\uzn$. Hence, for every $u\in BV(Q^\ell(I-z);\M)$ we have
\begin{equation*}
u=\uzn\text{  near }\partial Q^\ell(I-z)\ \Longleftrightarrow u_z=\uzn\text{ near }\partial Q^\ell(I),
\end{equation*}
while the stationarity of $g$ with respect to $\{\tau_z\}_{z\in\R^d}$ together with a change of variables yields
\begin{equation*}
\Egw (u,Q^\ell(I-z))=E_g(\tau_z^\nu\omega)(u_z,Q^\ell(I)).
\end{equation*}
Thus, we conclude by minimization that $\mu_{\ell}(I-z,\omega)=\mu_{\ell}(I,\tau_z^\nu\omega)$, which implies the stationarity of the process with respect to the lower-dimensional group action $\{\tau_z^\nu\}_{z\in\mathbb{R}^{d-1}}$. 

We conclude this step by showing that $\mu_\ell$ is subadditive. Let $I\in\mathcal{I}$ and let $(I^i)_{i=1}^k\subset\mathcal{I}$ be pairwise disjoint and such that $I=\bigcup_{i=1}^k I^i$. Fix $\eta>0$ and for any $i\in\{1,\ldots, k\}$ let $u^i\in \Adm(\uzn,Q^\ell(I^i))$ be such that 
\begin{align}\label{subad:almost optimal}
\Egw (u^i, Q^\ell(I^i))\leq \mu_{\ell}(I^i,\omega)+k^{-1}\eta.
\end{align}
Note that also the $d$-dimensional cuboids $Q^{\ell}(I^i)$ are pairwise disjoint, so that we can define a function $u\in BV(Q^{\ell}(I);\M)$ by setting 
\begin{equation*}
u(x):=
\begin{cases}
u^i(x) &\text{if}\ x\in Q^{\ell}(I^i)\ \text{for some}\ 1\leq i\leq k,\\
\uzn(x) &\text{otherwise}.
\end{cases}
\end{equation*}
Since $s_{\rm max}(I^i)\leq s_{\rm max}(I)$ for all $1\leq i\leq k$, all cuboids $Q^{\ell}(I^i)$ are contained in $Q^{\ell}(I)$, so that the function $u$ satisfies $u=\uzn$ near $\partial Q^\ell(I)$. Moreover, since $S_{\uzn}=H^\nu$, thanks to the boundary conditions satisfied by each $u^i$ and the equality $I=\bigcup_{i=1}^kI^i$ we have $S_u\cap Q^\ell(I)=S_u\cap\big(\bigcup_{i=1}^k\overline{Q^\ell(I^i)}\big)=\bigcup_{i=1}^kS_{u^i}$. Thus, using the additivity of $\Egw$ as a set function, from~\eqref{subad:almost optimal} we infer
\begin{align*}
\mu_\ell(I,\omega)\leq \Egw (u,Q^{\ell}(I))=\sum_{i=1}^k \Egw (u^i,Q^\ell(I^i))\leq\mu_\ell(I^i,\omega)+\eta.
\end{align*}
We then obtain the subadditivity of the process by the arbitrariness of $\eta>0$.

\medskip

\textbf{Step 2.} Existence of the limit 
\\
Suppose first that $\ell\in\N$; for every $t>0$ consider the sets $I_t\defas[-\frac{t}{2},\frac{t}{2})^{d-1}$. Thanks to the first equality in~\eqref{relation:mu-ell-X-t-ell} and Step 1 we can apply the multi-parameter subadditive ergodic theorem (cf. \cite[Theorem 3.11]{CDMSZ19} which is a slightly improved version of \cite[Theorem 2.4]{AkKr}) to deduce the existence of a set $\Omega_\ell^{a,b,\nu}$ of full probability and a function $g_{\ell}(\cdot,a,b,\nu):\Omega\to[0,+\infty)$ such that for every $\omega\in\Omega_\ell^{a,b,\nu}$ we have
\begin{align}\label{ex:limit:int}
g_{\ell}(\omega,a,b,\nu)=\lim_{t\to+\infty}\frac{\mu_\ell(I_t,\w)}{t^{d-1}}=\lim_{t\to +\infty}X_{t,\ell}^{a,b,\nu}(g)(\omega).
\end{align}
If instead $\ell=+\infty$, then the second equality in~\eqref{relation:mu-ell-X-t-ell} is valid for all $t>0$ and we obtain a set $\Omega_\infty^{a,b,\nu}$ of full probability and a function $g_\infty(\cdot,a,b,\nu):\Omega\to[0,+\infty)$ such that for every $\omega\in\Omega_\infty^{a,b,\nu}$ we have
\begin{align}\label{ex:limit:int-infinity}
g_\infty(\omega,a,b,\nu)=\lim_{t\to+\infty}X_{t,t}^{a,b,\nu}(g)(\omega).
\end{align}
%
%We just sketch how to pass from the particular sequence $k\mapsto 2mk$ to general sequences: let $t_k\to+\infty$ be arbitrary and set $t^-_{k}:=2m\lfloor \frac{t_k}{2m}\rfloor$, $t^+_{k}:=2m(\lfloor \frac{t_k}{2m}\rfloor+1)$. By a classical extension argument via the boundary conditions one can argue similar to the proof of subadditivity to show that
%%
%\begin{equation}\label{est:nonint1}
%X_{t_k^+,m\ell}^{\zeta,\nu}(g)(\omega)\leq X_{t_k,m\ell}^{\zeta,\nu}(g)(\omega)+C\frac{(t^+_k)^{d-2}}{(t^+_k)^{d-1}}\leq X_{t_k^-,m\ell}^{\zeta,\nu}(g)(\omega)+C\left(\frac{(t^+_k)^{d-2}}{(t^+_k)^{d-1}}+\frac{t_k^{d-2}}{t_k^{d-1}}\right).
%\end{equation}
%%
%Combining \eqref{ex:limit:int} and \eqref{est:nonint1} we infer that
%%
%\begin{align*}
%\limsup_{k\to+\infty}X_{t_k,m\ell}^{\zeta,\nu}\leq g_{\ell}(\omega,\zeta,\nu)\leq\liminf_{k\to+\infty}X_{t_k,m\ell}^{\zeta,\nu}(g)(\w).
%\end{align*}
%%
%Since the sequence $(t_k)$ was arbitrary we obtain~\eqref{ex:limit:rational}. 
Gathering~\eqref{ex:limit:int} and~\eqref{ex:limit:int-infinity} we get~\eqref{ex:limit:rational} and it remains to show the shift invariance.
To this end, fix $\omega\in\Omega_{\ell}^{a,b,\nu}$ and $z\in\R^{d-1}$, let $t\to+\infty$, and set $t^\pm\defas t\pm 2|z^\nu|$. Then, thanks to the stationarity of $g$, by an extension argument as in~\eqref{est:monotonicity-t} on can show that
\begin{equation}\label{est:extension}
X_{t^+,\ell}^{a,b,\nu}(g)(\omega)\leq X_{t,\ell}^{a,b,\nu}(g)(\tau_z^\nu\omega)+\frac{c(t^+-t)}{t^+}\leq X_{t^-,\ell}^{a,b,\nu}(g)(\omega)+\frac{c(t^+-t)}{t^+}+\frac{c(t-t^-)}{t}.
\end{equation}
Thus, from~\eqref{ex:limit:int} and~\eqref{ex:limit:int-infinity}, respectively, we deduce that
\begin{equation*}
\lim_{t\to+\infty}X_{t,\ell}^{a,b,\nu}(g)(\tau_z^\nu\omega)=g_\ell(\omega,a,b,\nu),\quad\lim_{t\to+\infty}X_{t,t}^{a,b,\nu}(g)(\tau_z^\nu,\omega)=g_\infty(\omega,a,b,\nu),
\end{equation*}
that is, $\tau_z^\nu\omega\in\Omega_{\ell}^{a,b,\nu}$ and~\eqref{eq:partialinvariance} holds true.
%
%that for all $\w$ belonging to the set of full measure $\widehat{\Omega}:=\cap_{\substack{\nu\in \mathbb{S}^{d-1}\cap \mathbb{Q}^d\\ \zeta\in\mathcal{M}}}\Omega^{\zeta,\nu}$, for every $\nu\in \mathbb{S}^{d-1}\cap\mathbb{Q}^d$, every $\zeta\in\mathcal{M}$, and for all $\ell\in\N$ there exists the limit
%%
%\begin{align}\label{ex:limit:rational}
%g_{\ell}(\w,\zeta,\nu)=\lim_{t\to+\infty}X_{t,m\ell}^{\zeta,\nu}(g)(\w).
%\end{align}
%%
%Finally, we remark that as a consequence of the ergodic theorem the limit $g_{\ell}(\cdot,\zeta,\nu)$ is invariant under the group action $(\tau_z^\nu)_{z\in\Z^{d-1}}$,\ie for a.e. $\w\in\Omega$ and every $z\in\Z^{d-1}$ it holds that
%\begin{equation}\label{eq:partialinvariance}
%g_{\ell}(\tau_z^\nu\w,\zeta,\nu)=g_{\ell}(\omega,\zeta,\nu).
%\end{equation} 
\end{proof}
%
%\begin{remark}\label{r.ex:limit:rational}
%Let $(a,b)\in\M\times\M$ and $\nu\in\Sph^{d-1}\cap\Q^{d\times d}$; under the assumptions of Lemma~\ref{l.ellfixed} we have that for all $\omega$ belonging to the set of full probability $\Omega^{a,b,\nu}\defas\bigcap_{\ell=0}^{\infty}\Omega_{\ell}^{\zeta,\nu}$ there exists
%%
%\begin{equation}\label{ex:limit:allell}
%g_{\ell}(\omega,a,b,\nu)=\lim_{t\to+\infty}X_{t,\ell}^{a,b,\nu}(g)(\omega)\quad\text{for every}\ \ell\in\N\cup\{+\infty\}.
%\end{equation}
%
%Moreover, setting 
%%
%\begin{equation*}
%\widetilde{\Omega}\defas\bigcap_{\substack{\zeta\in\M\\\nu\in\Sph^{d-1}\cap\Q^{d\times d}}}\Omega^{\zeta,\nu}
%\end{equation*}
%%
%we clearly have $\P(\widetilde{\Omega})=1$ and for all $\omega\in\widetilde{\Omega}$,~\eqref{ex:limit:allell} is satisfied for all $\zeta\in\M$ and $\nu\in\Sph^{d-1}$.
%\end{remark}
Based on Lemma~\ref{l.ellfixed} we now prove Theorem~\ref{t.compressedprocess}. Namely, having at hand the almost sure existence of the limits in~\eqref{ex:limit:rational} we aim to show that we can switch the limit as $t\to+\infty$ and $\ell\to+\infty$ at least in conditional expectation to obtain~\eqref{ex:limit:compressed}.
This will be done by exploiting suitable monotonicity properties of $g_\ell$ and $\mu_\ell$ together with the group-invariance of conditional expectations with respect to the $\sigma$-algebra of $\tau$-invariant sets. The latter property might be well-known in probability theory but we include the short proof for the sake of completeness. Recall that the conditional expectation of a random variable $X\in L^1(\Omega)$ with respect to a $\sigma$-algebra $\mathcal{F}'\subset\mathcal{F}$ is the  (almost surely) uniquely defined $\mathcal{F}'$-measurable function $\mathbb{E}[X|\mathcal{F}']$ such that for any $F\in\mathcal{F}'$ we have
\begin{equation*}
\int_FX\,\mathrm{d}\mathbb{P}=\int_{F}\mathbb{E}[X|\mathcal{F}']\,\mathrm{d}\mathbb{P}.
\end{equation*}
\begin{lemma}\label{l.invariance}
Let $X\in L^1(\Omega)$ be a random variable and $\widetilde{\tau}:\Omega\to\Omega$ be a measurable, measure-preserving map. Let $\mathcal{F}_1\subset\mathcal{F}$ be a $\sigma$-algebra containing only $\widetilde{\tau}$-invariant sets,\ie $\mathbb{P}(\widetilde{\tau}(F)\Delta F)=0$ for all $F\in\mathcal{F}_1$.
Then
$$\mathbb{E}[X\circ\widetilde{\tau}|\mathcal{F}_1]=\mathbb{E}[X|\mathcal{F}_1].$$ %$\mathbb{E}[X|\mathcal{F}_1]\circ\tau_z=\mathbb{E}[X|\mathcal{F}_1]$.
\end{lemma}
\begin{proof}
First note that since $\widetilde{\tau}$ is measure preserving it follows that $X\circ\widetilde{\tau}\in L^1(\Omega)$. Hence $\mathbb{E}[X\circ\widetilde{\tau}|\mathcal{F}_1]$ is well-defined and in particular $\mathcal{F}_1$-measurable. Next fix $F\in\mathcal{F}_1$. By a change of variables and $\widetilde{\tau}$-invariance of $F$ we have
\begin{equation*}
\int_F\mathbb{E}[X\circ\widetilde{\tau}|\mathcal{F}_1](\w)\,\mathrm{d}\mathbb{P}=\int_{F}X(\widetilde{\tau}(\w))\,\mathrm{d}\mathbb{P}=\int_{\widetilde{\tau} (F)}X(\w)\,\mathrm{d}\mathbb{P}=\int_FX(\w)\,\mathrm{d}\mathbb{P},
\end{equation*}
where the first equality follows from the definition of the conditional expectation. This proves the claim.
\end{proof}
\begin{proof}[Proof of Theorem~\ref{t.compressedprocess}]
The first part of the proof is an immediate consequence of Lemma~\ref{l.ellfixed} in the case $\ell=+\infty$, Remark~\ref{r.ex:limit:rational} and the fact that for every $(a,b)\in\M\times\M$ and $\omega\in\Omega$ the restrictions of the mappings $\nu\mapsto\liminf_{t} X_{t,t}^{a,b,\nu}(g)(\omega)$, $\nu\mapsto \limsup_{t} X_{t,t}^{a,b,\nu}(g)(\omega)$ to $\Sd\setminus\{-e_d\}$ are continuous (the latter can be shown arguing word by word as in~\cite[Lemma 5.5]{CDMSZ19}). Namely, by setting
\begin{align*}
\widehat{\Omega}\defas\bigcap_{\substack{(a,b)\in\M\times\M\\\nu\in\Sph^{d-1}\cap\mathbb{Q}^{d\times d}}}\Omega_{\infty}^{a,b,\nu}\quad\text{and}\quad g_{\rm hom}(\cdot,a,b,\nu)\defas g_\infty(\cdot,a,b,\nu),
\end{align*}
we clearly have that $\mathbb{P}(\widehat{\Omega})=1$, while~\eqref{ex:limit:rational} together the above mentioned continuity ensures that~\eqref{ex:limit:cubes} holds true for every $(a,b)\in\M\times\M$, $\nu\in\Sph^{d-1}$ and $\omega\in\widehat{\Omega}$. Eventually, in view of Remark~\ref{r.ex:limit:rational}, $g_{\rm hom}=g_\infty$ satisfies the required invariance property and hence is deterministic in the case that $g$ is ergodic.

\medskip
We now come to prove the second part of Theorem~\ref{t.compressedprocess}. 
Let $\nu\in\Sd$ be fixed, $\Omega^\nu$ as in Remark~\ref{r.ex:limit:rational} and for every $(a,b)\in\M\times\M$ and $\ell\in\N\cup\{+\infty\}$ let $g_\ell(\omega,a,b,\nu)$ be as in Lemma~\ref{l.ellfixed}. The main part of the proof consists in showing that there exists a set $\widehat{\Omega}^{\nu}\subset\Omega^\nu$ of full probability such that
\begin{equation}\label{claim:limitell}
\lim_{\ell\to+\infty}g_\ell(\omega,a,b,\nu)=g_{\infty}(\omega,a,b,\nu)\quad\text{for every}\ \omega\in\widehat{\Omega}^{\nu}.
\end{equation}
To this end, we fix $(a,b)\in\M\times\M$ and we compute $g_{\ell}(\w,a,b,\nu)$ using $\mu_\ell([-2^k,2^k),\omega)$, where we use the shorthand $\mu_\ell\defas\mu_\ell^{a,b,\nu}$. First note that \eqref{eq:partialinvariance} implies that $g_{\ell}(\cdot,a,b,\nu)$ is measurable with respect to the $\sigma$-algebra of $\{\tau^\nu_{z}\}_{z\in\mathbb{\R}^{d-1}}$-invariant sets defined by
\begin{equation*}
\mathcal{F}^\nu=\{F\in\mathcal{F}:\,\mathbb{P}((\tau^\nu_zF)\Delta F)=0\quad\forall\, z\in\R^{d-1}\}.
\end{equation*}
Hence, $g_{\ell}(\w,a,b,\nu)=\mathbb{E}[g_{\ell}(\cdot,a,b,\nu)|\mathcal{F}^\nu](\w)$ for a.e. $\w$. Together with the uniform bound~\eqref{eq:l_inftybound} and the dominated convergence theorem for the conditional expectation this ensures that almost surely we have
\begin{equation}\label{eq:gell}
g_{\ell}(\w,a,b,\nu)=\mathbb{E}[g_{\ell}(\cdot,a,b,\nu)|\mathcal{F}^\nu ](\w)=\lim_{k\to +\infty}\frac{\mathbb{E}[\mu_{\ell}([-2^k,2^k)^{d-1},\cdot)|\mathcal{F}^\nu ](\w)}{(2^{k+1})^{(d-1)}}.
\end{equation}
Thanks to~\eqref{est:monotonicity-ell} this in turn implies that almost surely
\begin{equation}\label{eq:inf-ell-lim-k}
\lim_{\ell\to+\infty}g_{\ell}(\w,a,b,\nu)=\inf_{\ell\in\N}\lim_{k\to +\infty}\frac{\mathbb{E}[\mu_{\ell}([-2^k,2^k)^{d-1},\cdot)|\mathcal{F}^\nu ](\w)}{(2^{k+1})^{(d-1)}}.
\end{equation}
We now show that also the limit in $k$ in~\eqref{eq:inf-ell-lim-k} is an infimum, which will then allow us to switch the infima in $\ell$ and $k$. 
To this end, we notice that for every $k\in\N$ the cube $[-2^{k+1},2^{k+1})^{d-1}$ can be partitioned into $n_d:=2^{d-1}$ integer-translated disjoint copies of $[-2^k,2^k)^{d-1}$,\ie there exist $z_1,\ldots,z_{n_d}\in\mathbb{Z}^{d-1}$ with 
\begin{equation*}
[-2^{k+1},2^{k+1})^{d-1}=\bigcup_{n=1}^{n_d}\big(z_n+[-2^k,2^k)^{d-1}\big),\quad \left(z_n+[-2^k,2^k)^{d-1}\right)\cap \left(z_m+[-2^k,2^k)^{d-1}\right)=\emptyset\; \text{ for }\; n\neq m. 
\end{equation*}
Hence, using the subadditivity of the stochastic process $\mu_{\ell}$ and its $\{\tau_z^\nu\}_{z\in\R^{d-1}}$-stationarity we can write
\begin{equation*}
\mu_{\ell}([-2^{k+1},2^{k+1})^{d-1},\w)\leq \sum_{n=1}^{n_d}\mu_{\ell}([-2^k,2^k)^{d-1},\tau^\nu_{z_n}\w).
\end{equation*}
Taking the conditional expectation with respect to the $\sigma$-algebra $\mathcal{F}^\nu $ and using that it is linear and order preserving, by Lemma \ref{l.invariance} we find that almost surely
\begin{align*}
\mathbb{E}[\mu_{\ell}([-2^{k+1},2^{k+1})^{d-1},\cdot)|\mathcal{F}^\nu ](\w)\leq&\sum_{n=1}^{n_d}\mathbb{E}[\mu_\ell([-2^k,2^k)^{d-1},\cdot)\circ \tau_{z_n}^\nu|\mathcal{F}^{\nu}](\w)
\\
=& n_d\mathbb{E}[\mu_{\ell}([-2^k,2^k)^{d-1},\cdot)|\mathcal{F}^\nu ](\w).
\end{align*}
Dividing this estimate by $(2^{k+2})^{(d-1)}$ we see that the map $k\mapsto \frac{\mathbb{E}[\mu_{\ell}([-2^{k},2^{k})^{d-1},\cdot)|\mathcal{F}^\nu ](\w)}{(2^{k+1})^{(d-1)}}$ is almost surely decreasing. Together with~\eqref{eq:gell} this implies that almost surely
\begin{equation}\label{eq:inf:k}
g_\ell(\omega,a,b,\nu)=\inf_{k\in\N}\frac{\mathbb{E}[\mu_{\ell}([-2^k,2^k)^{d-1},\cdot)|\mathcal{F}^\nu ](\w)}{(2^{k+1})^{(d-1)}}\quad\text{for every}\ \ell\in\N\cup\{+\infty\}.
\end{equation}
Eventually, from~\eqref{est:monotonicity-ell} and~\eqref{relation:mu-ell-X-t-ell} we infer that also $\ell\mapsto\mu_\ell([-2^k,2^k),\omega)$ is almost surely decreasing. Since the monotonicity is preserved by the conditional expectation, gathering~\eqref{eq:inf-ell-lim-k} and~\eqref{eq:inf-ell-lim-k} we get
\begin{equation}\label{inf:limit:ell}
\lim_{\ell\to +\infty}g_\ell(\omega,a,b,\nu)=\inf_{
k,\ell\in\N}\frac{\mathbb{E}[\mu_{\ell}([-2^k,2^k)^{d-1},\cdot)|\mathcal{F}^\nu ](\w)}{(2^{k+1})^{(d-1)}}=\inf_{k\in\N}\lim_{\ell\to+\infty}\frac{\mathbb{E}[\mu_{\ell}([-2^k,2^k)^{d-1},\cdot)|\mathcal{F}^\nu ](\w)}{(2^{k+1})^{(d-1)}}
\end{equation}
almost surely. 
For fixed $k\in\N$ we have $\mu_{\ell}([-2^k,2^k)^{d-1},\omega)=\mu_\infty([-2^k,2^k)^{d-1},\omega)$ for $\ell> 2^{k+1}$. 
Thus, for fixed $k\in\N$, on the right-hand side of~\eqref{inf:limit:ell} we can pass to the limit with respect to $\ell$ inside the conditional expectation using again the corresponding dominated convergence theorem to deduce that almost surely
\begin{equation*}
\lim_{\ell\to +\infty}g_{\ell}(\w,a,b,\nu)=\inf_{k\in\N}\frac{\mathbb{E}[\mu_{\infty}([-2^k,2^k)^{d-1},\cdot)|\mathcal{F}^\nu ](\w)}{(2^{k+1})^{(d-1)}}=g_\infty(\omega,a,b,\nu)=g_{\rm hom}(\omega,a,b,\nu),
\end{equation*}
where the second equality follows from~\eqref{eq:inf:k} applied with $\ell=+\infty$. Hence, there exists $\widehat{\Omega}^{\nu}$ with full probability such that~\eqref{claim:limitell} is satisfied.

\medskip
Thanks to~\eqref{claim:limitell} we are now able to conclude as follows.
We fix sequences $0<\ell_t\leq t$ such that $\ell_t\to +\infty$ as $t\to +\infty$. Then for any fixed $\ell\in\N$ and $t$ large~\eqref{est:monotonicity-ell} implies that
\begin{equation*}
X_{t,t}^{a,b,\nu}(g)(\w)\leq X_{t,\ell_t}^{a,b,\nu}(g)(\w)\leq X_{t,\ell}^{a,b,\nu}(g)(\omega).
\end{equation*}
Passing to the limit in $t\to +\infty$, the left-hand side converges to $g_{\rm hom}(\w,a,b,\nu)$ for $\omega\in\Omega^\nu$, while for the right-hand side we use \eqref{ex:limit:rational}, so that
\begin{equation*}
g_{\rm hom}(\w,\a,b,\nu)\leq \liminf_{t\to +\infty}X_{t,\ell_t}^{a,b,\nu}(g)(\w)\leq\limsup_{t\to +\infty}X_{t,\ell_t}^{a,b,\nu}(g)(\w)\leq g_{\ell}(\w,a,b,\nu),
\end{equation*}
for every $\omega\in\Omega^\nu$. Thus, letting $\ell\to +\infty$ and using~\eqref{claim:limitell} we deduce that for all $\omega\in\widehat\Omega^{\nu}$ we have~\eqref{ex:limit:compressed}, hence Theorem~\ref{t.compressedprocess} is proved.
\end{proof}
\subsection{Estimating the oscillation and concentration inequalities: Proof of Theorem~\ref{t.concentration}}
\begin{proof}[Proof of Theorem \ref{t.concentration}]
Throughout the proof $(a,b)\in\mathcal{M}\times\M$ and $\nu\in \mathbb{S}^{d-1}$ will be fixed, so we write $X_{t,\ell_t}(g)=X_{t,\ell_t}^{a,b,\nu}(g)$ to reduce notation. Moreover, for every $U\subset\R^d$ let the quantity $\partial^{\rm osc}_{g,U}X_{t,\ell_t}(g)$ be as in~\eqref{def:oscillation} with $X=X_{t,\ell_t}$. Let $t\geq\ell_t\geq 1$ be fixed; 
thanks to \cite[Proposition 1.10]{DG1} and Assumption \ref{ass:2}
there exists a constant $C>0$ such that for every $p\geq 1$ we have the estimate
\begin{align}\label{est:spectralgap}
\E[(X_{t,\ell_t}(g)-\E[X_{t,\ell_t}(g)])^{2p}]\leq(Cp^2)^p\E\Bigg[\int_0^{+\infty}\Bigg(\int_{\R^d}\Big(\partial^{\rm osc}_{g, B_{2(s+1)}(x)}X_{t,\ell_t}(g)\Big)^2\dx\Bigg)^{p}(s+1)^{-dp}\pi(s)\,\mathrm{d}s\Bigg].
\end{align}
Thus, to obtain~\eqref{est:variance} we fix $\omega\in\Omega$ and we suitably bound the term
\begin{align}\label{spectralgap}
\int_0^{+\infty}\Bigg(\int_{\R^d}\Big(\partial^{\rm osc}_{g, B_{2(s+1)}(x)}X_{t,\ell_t}(g) (\omega) \Big)^2\dx\Bigg)^{p}(s+1)^{-dp}\pi(s)\ds.
\end{align}
We first show that in~\eqref{spectralgap} we can reduce the domain of integration in $x$.
In fact, for given $s>0$ suppose that $x\in\R^d$ is such that 
\begin{equation}\label{cond:x}
B_{2(s+1)}(x)\cap R^\nu_{t,\ell_t}=\emptyset
\end{equation}
and let $g'\in\A_c$ with $g'=g(\w)$ on $\R^d\setminus B_{2(s+1)}(x)\times\R^m\times\Sd$. Then clearly $X_{t,\ell_t}(g')=X_{t,\ell_t}(g)(\omega)$, which implies that $\partial^{\rm osc}_{g,B_{2(s+1)}(x)} X_{t,\ell_t}(g)=0$. Thus, since~\eqref{cond:x} is satisfied for all $x\in\R^d\setminus R^\nu(s,t)$ with 
%Note that for every $\zeta\in\ps_{r,R}$ and every $x,y\in\zeta$ it holds by assumption that
%%
%\begin{equation}\label{eq:neighbordistance}
%(x,y)\in\mathcal{E}(\zeta)\implies |x-y|\leq 2R.
%\end{equation}
%%
%Hence we have $\mathcal{E}(\zeta)|_{R_{\nu}(t,\ell_t)}=\mathcal{E}(\xi)|_{R_{\nu}(t,\ell_t)}$, which implies that $X_{t,\ell_t}(\zeta)=X_{t,\ell_t}(\xi)$ and hence for such $x$ it holds that $\partial^{\rm osc}_{\xi, B_{2(s+1)}(x)}X_{t,\ell_t}(\xi)=0$. 
%Thus, the domain of integration for $x$ in \eqref{spectralgap} can be reduced to 
\begin{equation*}
R^\nu(s,t)\defas R^\nu_{t+4(s+1),\ell_t +4(s+1)},
\end{equation*}
we have
\begin{align}\label{integral:osc}
\int_{\R^d}\Big(\partial^{\rm osc}_{g, B_{2(s+1)}(x)}X_{t,\ell_t}(g) (\omega) \Big)^2\dx=\int_{R^\nu(s,t)}\Big(\partial^{\rm osc}_{g, B_{2(s+1)}(x)}X_{t,\ell_t}(g)(\omega)\Big)^2\dx.
\end{align}
We show that there exists a dimensional constant $C=C_d>0$ such that for all $t\geq \ell_t\geq 1$ and for all $s>0$ and $x\in R^\nu(s,t)$ we have
\begin{align}\label{est:osc}
\partial^{\rm osc}_{g,B_{2(s+1)}(x)}X_{t,\ell_t}(g)\leq C\Big(\frac{s+1}{t}\Big)^{d-1}.
\end{align}
We distinguish the following two exhaustive cases:
\begin{enumerate}[label=(\alph*)]
\item\label{inclusion} $s>0$ and $x\in R^\nu(s,t)$ are such that $R^\nu_{t,\ell_t}\subset B_{2(s+1)}(x)$;
\item\label{intersection} $s>0$ and $x\in R^\nu(s,t)$ are such that $R^\nu_{t,\ell_t}\cap(\R^d\setminus B_{2(s+1)}(x))\neq\emptyset$.
\end{enumerate}
Suppose that we are in the case~\ref{inclusion} and let $g'\in\A_c$ be such that $g'=g$ on $\R^d\setminus B_{2(s+1)}(x)\times\R^m\times\Sd$. Then, to obtain \eqref{est:osc} it suffices to use~\eqref{est:oscillation-trivial}, since the inclusion $R^\nu_{t,\ell_t}\subset B_{2(s+1)}(x)$ implies that $2(s+1)\geq t/2$, from which we readily deduce that
\begin{align*}
\partial^{\rm osc}_{g,B_{2(s+1)}(x)}X_{t,\ell_t}(g)\leq 2c\leq 4^{d-1}2c\Big(\frac{s+1}{t}\Big)^{d-1}.
\end{align*}
We now prove \eqref{est:osc} in the case \ref{intersection}, where the above construction is not optimal. Instead, we choose $u\in BV(R^\nu_{t,\ell_t};\M)$ such that $u=\uzn$ near $\partial R^\nu_{t,\ell_t}$ and
\begin{align}\label{est:almost-optimal}
\int_{S_u\cap R^\nu_{t,\ell_t}}g(\w,y,u^+-u^-,\nu_u)\dHd\leq t^{d-1}X_{t,\ell_t}(g)(\w)+1.
\end{align}
Then we define a new function $\tilde{u}\in BV(R^\nu_{t,\ell_t};\M)$ by setting
\begin{align*}
\tilde{u}:=
\begin{cases}
u &\text{in }R^\nu_{t,\ell_t}\setminus \overline{B}_{2(s+1)}(x),\\
\uzn &\text{in } R^\nu_{t,\ell_t}\cap \overline{B}_{2(s+1)}(x).
\end{cases}
\end{align*}
By construction $\tilde{u}=\uzn$ near $\partial R^\nu_{t,\ell_t}$, hence for every $g'\in\A_c$ with $g'=g$ on $\R^d\setminus B_{2(s+1)}(x)\times\R^m\times\Sd$ we have
\begin{align}\label{est:Xgprime}
t^{d-1}X_{t,\ell_t}(g') &\leq\int_{S_{\tilde{u}}\cap R^\nu_{t,\ell_t}} g'(y,\tilde{u}^+-\tilde{u}^-,\nu_{\tilde{u}})\dHd\nonumber
\\
&\leq\int_{S_u\cap R^\nu_{t,\ell_t}}g(\w,y,u^+-u^-,\nu_u)\dHd+c\Hd\big(S_{\tilde{u}}\cap\overline{B}_{2(s+1)}(x)\big)\,.
\end{align}
By construction, we have
\begin{align}\label{est:measure-jumpset}
\Hd\big(S_{\tilde{u}}\cap\overline{B}_{2(s+1)}(x)\big)&\leq\Hd(\Hn\cap\overline{B}_{2(s+1)})+\Hd(\partial B_{2(s+1)})\nonumber
\\
&\leq{\rm diam}(B_{2(s+1)})^{d-1}+\Hd(\partial B_{2(s+1)})\leq C(s+1)^{d-1},
\end{align}
for some $C=C_d>0$. Thus, gathering~\eqref{est:almost-optimal}, \eqref{est:Xgprime}, and \eqref{est:measure-jumpset}, and dividing by $t^{d-1}$, we obtain
\begin{align*}
X_{t,\ell_t}(g')\leq  X_{t,\ell_t}(g)(\omega)+C\Big(\frac{s+1}{t}\Big)^{d-1}
\end{align*}
with $C$ only depending on $d$. Then,~\eqref{est:osc} follows by passing to the supremum in $g'$ and using the triangular inequality.

Using~\eqref{est:osc} we can estimate the right-hand side in~\eqref{integral:osc} via
\begin{equation}\label{est:integral1}
\int_{R^\nu(s,t)}\Big(\partial^{\rm osc}_{g, B_{2(s+1)}(x)}X_{t,\ell_t}(g)(\w)\Big)^2\dx\leq 
C\Big(\frac{s+1}{t}\Big)^{2(d-1)}|R^\nu(s,t)|.
\end{equation}
Moreover, since $1\leq\ell_t\leq t$, we can bound the volume $|R^\nu(s,t)|=(t+4(s+1))^{d-1}(\ell_t+4(s+1))$ via
\begin{align*}
|R^\nu(s,t)|\leq C(s+1)^dt^{d-1}\ell_t%\leq 2^{3d}(s+1)^dt^{d-1}\ell_t
\,,
\end{align*}
hence the integral in~\eqref{est:integral1} can be further estimated via
\begin{align*}
\int_{R^\nu(s,t)}\Big(\partial^{\rm osc}_{g, B_{2(s+1)}(x)}X_{t,\ell_t}(g)(\w)\Big)^2\dx\leq Ct^{1-d}\ell_t (s+1)^{3d-2}.
\end{align*}
Eventually, combining the above inequality with~\eqref{integral:osc} and~\eqref{est:spectralgap} and integrating over $s\in(0,+\infty)$ and $\omega\in\Omega$ we obtain
\begin{align*}
\E[|X_{t,\ell_t}(g)-\E[X_{t,\ell_t}(g)]|^{2p}]\leq (c_d p^2)^p t^{p(1-d)}\ell_t^p\int_0^{+\infty}(s+1)^{2p(d-1)}\pi(s)\ds
\end{align*}
with $c_d>0$.
%\begin{align*}
%&c_{d,r}^{2p}3^{pd-2}t^{p(1-d+\alpha)}\int_0^{+\infty}s^{p(d-2)}\exp\Big(-\frac{s}{c_r}\Big)\ds\\
%&+c_{dr}^{2p}3^{pd-2}2^pt^{p(1-d)}\int_0^{+\infty}\Big((s+1)^p+2^pr^p\Big)s^{p(d-2)}\exp\Big(-\frac{s}{c_r}\Big)\ds\\
%&+c_{dr}^{2p}3^{pd-2}2^{p(d-1)}t^{p(\alpha-2d+2)}\int_0^{+\infty}\Big((s+1)^{p(d-1)}+2^{p(d-1)}r^{p(d-1)}\Big)s^{p(d-2)}\exp\Big(-\frac{s}{c_r}\Big)\ds\\
%&+c_{dr}^{2p}3^{pd-2}2^{pd}t^{2p(1-d)}\int_0^{+\infty}\Big((s+1)^{p(d-1)}+2^{p(d-1)}r^{p(d-1)}\Big)\Big((s+1)^p+2^pr^p\Big)s^{p(d-2)}\exp\Big(-\frac{s}{c_r}\Big)\ds.
%\end{align*}
%
\end{proof}
%\ab{Observations/ Remarks
%\begin{itemize}
%\item for $d\geq 2$ and $\alpha\in (0,1)$ all exponents in $t$ are negative
%\item for $p>\frac{1}{d-1-\alpha}$ we get summability in $t$
%\item I did not compute the integrals in $s$ since I don't know what we need exactly. However, they are finite (and may be computed using the $\Gamma$-function, in particular they should behave roughly like $p!$)
%\end{itemize}
%}
Finally, we derive strong concentration estimates in the case of an exponentially decaying weight $\pi$.
\begin{proof}[Proof of Corollary \ref{c.expweight}]
	We first bound the integral appearing in Theorem \ref{t.concentration}. Without loss of generality we can assume that $\pi(s)=C\exp(-\tfrac{s}{C})$ with $C\geq 1$. Using two changes of variables we have
	\begin{align*}
		\int_0^{+\infty}(s+1)^{2p(d-1)}\pi(s)\ds\overset{s+1=y}{=}&C\int_1^{+\infty}y^{2p(d-1)}\exp(-\tfrac{y}{C})\exp(\tfrac{1}{C})\,\mathrm{d}y
		\\
		\overset{y/C=x}{\leq} &C^{2p(d-1)+2}\exp(\tfrac{1}{C}) \int_{0}^{+\infty}x^{2p(d-1)}\exp(-x)\,\mathrm{d}x= C_d^p\,\Gamma(2p(d-1)+1).
	\end{align*}
	In what follows the constant $C_d$ may change, but will only depend on $d$. We use the following elementary bound on the $\Gamma$-function: $\Gamma(x+1)\leq 2\left(\frac{2x}{e}\right)^x$ for all $x>0$. Hence we can estimate the last factor by
	\begin{equation*}
		\Gamma(2p(d-1)+1)\leq  2\left(\frac{4p(d-1)}{ e}\right)^{2p(d-1)}=C_d^p p^{2p(d-1)}.
	\end{equation*}
	Hence from Theorem \ref{t.concentration} we infer that (upon increasing the dimensional constant $C_d$)
	\begin{equation}\label{est:expectation1}
		\mathbb{E}\left[\left(\frac{1}{C_d}\left|X_{t,\ell_t}^{a,b,\nu}(g)-\mathbb{E}[X_{t,\ell_t}^{a,b,\nu}(g)]\right|^{2}\right)^p\right]\leq p^{2pd}t^{p(1-d)}\ell_t^p\quad\text{for all $p\geq 1$.}
	\end{equation}
	We also need an estimate for $p\in [0,1)$. Since the function $x^x$ is bounded uniformly away from zero on $[0,1]$, we have $t^{p(1-d)}\ell_t^p\leq C_d p^{2pd}t^{p(1-d)}\ell_t^p$ upon further increasing $C_d$. Thus, applying Jensen's inequality with $p\in[0,1)$ and using~\eqref{est:expectation1} with $p=1$ leads to 
	\begin{align*}
		\mathbb{E}\left[\left(\frac{1}{C_d}\left|X_{t,\ell_t}^{a,b,\nu}(g)-\mathbb{E}[X_{t,\ell_t}^{a,b,\nu}(g)]\right|^{2}\right)^p\right]\leq & \mathbb{E}\left[\frac{1}{C_d}\left|X_{t,\ell_t}^{a,b,\nu}(g)-\mathbb{E}[X_{t,\ell_t}^{a,b,\nu}(g)]\right|^{2}\right]^p
		\\
		\leq &t^{p(1-d)}\ell_t^p
		\leq C_d p^{2pd}t^{p(1-d)}\ell_t^p.
	\end{align*}
	Therefore we conclude that
	\begin{equation*}
		\mathbb{E}\left[\left(\frac{1}{C_d}\left|X_{t,\ell_t}^{a,b,\nu}(g)-\mathbb{E}[X_{t,\ell_t}^{a,b,\nu}(g)]\right|^{2}\right)^p\right]\leq
		C_dp^{2pd}t^{p(1-d)}\ell_t^p\quad\forall\, p\geq 0.
	\end{equation*}
	For $n\in\N$ we set $p_n=\tfrac{n}{2d}$. Then the above estimate implies
	\begin{equation*}
		\mathbb{E}\left[\left(\frac{1}{C_d}\left|X_{t,\ell_t}^{a,b,\nu}(g)-\mathbb{E}[X_{t,\ell_t}^{a,b,\nu}(g)]\right|^{\frac{1}{d}}\right)^n\right]\leq
		C_d\left(\frac{n}{2d}\right)^{n}t^{n\frac{(1-d)}{2d}}\ell_t^{\frac{n}{2d}}\leq  C_d\, n!\,\left(\frac{3}{2d}\right)^n\left(t^{1-d}\ell_t\right)^{\tfrac{n}{2d}} \quad\forall\, n\in \N,
	\end{equation*}
	where we used the estimate $n^n\leq 3^n n!$ (valid for all $n\in\N$). For $n\in\N$ we can absorb the factor $C_d$ in the left-hand side. Dividing by $n!(t^{1-d}\ell_t)^\frac{n}{2d}$, summing the resulting estimate over $n\in\N$ and exchanging summation and expectation we deduce that
	\begin{equation}\label{est:expect1}
		\mathbb{E}\left[\exp\left(\frac{1}{C_d}\left|\frac{X_{t,\ell_t}^{a,b,\nu}(g)-\mathbb{E}[X_{t,\ell_t}^{a,b,\nu}(g)]}{\sqrt{t^{1-d}\ell_t}}\right|^{\frac{1}{d}}\right)\right]\leq  \sum_{n\geq 0} \left(\frac{3}{2d}\right)^n=\frac{1}{1-\left(\frac{3}{2d}\right)}\leq 4.
	\end{equation}
	This proves the first estimate in Theorem \ref{c.expweight}.
	To prove the second one we observe that~\eqref{ex:limit:compressed} together with~\eqref{eq:energybound} and the dominated convergence theorem implies that $\mathbb{E}[X_{t,\ell_t}^{a,b,\nu}(g)]\to g_{\rm hom}(a,b,\nu)$. Thus, for any $\eta>0$ we have
	\begin{equation}\label{est:concentration1}
	\mathbb{P}\left(|X_{t,\ell_t}^{a,b,\nu}(g)-g_{\rm hom}(a,b,\nu)|>\eta\right) \leq \mathbb{P}\left(|X_{t,\ell_t}^{a,b,\nu}(g)-\mathbb{E}[X_{t,\ell_t}^{a,b,\nu}(g)]|>\eta/2\right),
	\end{equation}
	for $t$ sufficiently large. Moreover, applying Markov's inequality and using~\eqref{est:expect1} we infer
	\begin{equation}\label{est:concentration2}
	\begin{split}
		\mathbb{P}\Big(|X_{t,\ell_t}^{a,b,\nu}(g) &-\mathbb{E}[X_{t,\ell_t}^{a,b,\nu}(g)]|>\eta/2\Big)\\
		&=\mathbb{P}\Bigg(\exp\bigg(\frac{1}{C_d}\bigg|\frac{X_{t,\ell_t}^{a,b,\nu}(g)-\mathbb{E}[X_{t,\ell_t}^{a,b,\nu}(g)]}{\sqrt{t^{1-d}\ell_t}}\bigg|^{\frac{1}{d}}\bigg)>\exp\bigg(\frac{1}{C_d}\bigg(\frac{\eta}{2\sqrt{t^{1-d}\ell_t}}\bigg)^{\tfrac{1}{d}}\bigg)\Bigg)
		\\
		&\leq 4 \exp\Bigg(-\frac{1}{C_d}\left(\frac{\eta}{2\sqrt{t^{1-d}\ell_t}}\right)^{\tfrac{1}{d}}\Bigg).
	\end{split}
	\end{equation}
	Hence, gathering~\eqref{est:concentration1}--\eqref{est:concentration2}, taking the logarithm, and passing to the limsup in $t$ we deduce that 
	\begin{align*}
		\limsup_{t\to +\infty}\left((t^{1-d}\ell_t)^{\tfrac{1}{2d}} \log\left(\mathbb{P}\left(|X_{t,\ell_t}^{a,b,\nu}(g)-g_{\rm hom}(a,b,\nu)|>\eta\right)\right)\right)\leq \limsup_{t\to +\infty}\,(t^{1-d}\ell_t)^{\tfrac{1}{2d}}\log(4)-\frac{1}{C_d}\left(\frac{\eta}{2}\right)^{\frac{1}{d}}.
	\end{align*}
	Note that $t^{1-d}\ell_t\leq t^{2-d}\leq 1$ for $t\geq \ell_t\geq 1$. If along the particular sequence $t^{1-d}\ell_t\to 0$, then we conclude by the above estimate. If the limsup is realized by a sequence such that $t^{1-d}\ell_t\geq c>0$, then the estimate in Corollary \ref{c.expweight} is trivial since $X_{t,\ell_t}^{a,b,\nu}(g)$ converges in probability to $g_{\rm hom}(a,b,\nu)$, so that the logarithmic term is negative for $t$ large enough. 
\end{proof}
\subsection{Almost plane-like formulas for $\Z^d$-stationary integrands}
In this subsection we show how to extend Theorem \ref{t.compressedprocess} to models with $\Z^d$-stationarity assuming quantitative concentration inequalities in form of a multi-scale functional inequality. 
\begin{proof}[Proof of Proposition \ref{p.expectation}]
	Fix $\nu\in\mathbb{S}^{d-1}$, $(a,b)\in\mathcal{M}\times\M$ and consider sequences $t_k,t_n'\to +\infty$ and $\ell_k,\ell_n'\to +\infty$ such that $0<\ell_k\leq t_k$ and $0<\ell_n'\leq t_n'$. For $n\in\N$ choose $K_0=K_0(n)\geq n\in\N$ such that for all $k\geq K_0$ we have $\ell_k\geq \ell'_n+2\sqrt{d}$ and $t_k\geq t_n'+2\sqrt{d}$. Consider the collection of integer vertices
	\begin{equation*}
		I_{n,k}=\{z\in \Z^{d-1}:\,Q_{n,z}:=t'_nz+(-\tfrac{t'_n}{2},\tfrac{t'_n}{2})^{d-1}\subset (-\tfrac{t_k-2\sqrt{d}}{2},\tfrac{t_k-2\sqrt{d}}{2})^{d-1}\}.
	\end{equation*}
	Note that with the orthogonal matrix $O_{\nu}$ introduced in \eqref{eq:matrix}, for every $z\in I_{n,k}$ it holds that \begin{equation}\label{eq:inclusion}
		O_{\nu}(Q_{n,z}\times\{0\})=t'_nO_{\nu}(z,0)+t'_n(Q^{\nu}\cap H^{\nu})\subset (t_k-2\sqrt{d})(Q^{\nu}\cap H^{\nu}).
	\end{equation} 
	Moreover, for $t'_n\geq 1$ we also infer that
	\begin{equation}\label{eq:almostfilling}
		\bigcup_{z\in I}O_{\nu}(Q_{n,z}\times\{0\})\supset (t_k-4\sqrt{d}t'_n)(Q^{\nu}\cap H^{\nu}).
	\end{equation}
	For each $z\in I_{n,k}$ we decompose the vector $t'_nO_{\nu}(z,0)\in H^{\nu}$ into its integer part and a remainder writing
	\begin{equation}\label{eq:decomposition}
	t'_nO_{\nu}(z,0)=z_n(z)-y_n(z),
	\end{equation}
	with $z_n(z)\in\Z^d$, $y_n(z)\in\R^d$ and $|y_n(z)|\leq\sqrt{d}$.
	 Then, combining~\eqref{eq:inclusion} and~\eqref{eq:decomposition} we obtain that
	\begin{equation}\label{eq:Qnz}
		O_{\nu}(Q_{n,z}\times\{0\})+y_n(z)=z_n(z)+t_n'(Q^{\nu}\cap H^{\nu}).
	\end{equation}
	Since $|y_n(z)|\leq\sqrt{d}$, it follows from \eqref{eq:inclusion} that $O_{\nu}(Q_{n,z}\times\{0\})+y_n(z)\subset t_k Q^{\nu}$ for all $z\in I_{n,k}$. In particular, since $\ell_k\geq \ell'_n+2\sqrt{d}$ we conclude that 
	\begin{equation}\label{eq:shiftfullinclusion}
		A_{n,\nu}(z):=O_{\nu}\left(Q_{n,z}\times \Big(-\tfrac{\ell_n'}{2},\tfrac{\ell_n'}{2}\Big)\right)+y_n(z)\subset\subset R^{\nu}_{t_k,\ell_k}.
	\end{equation} 
	Moreover, \eqref{eq:Qnz} implies that $A_{n,\nu}(z)=z_n(z)+R_{t'_n,\ell'_n}^{\nu}$ is an integer translate of $R^{\nu}_{t_n',\ell_n'}$, so that by stationarity of $g$ the random variables $Y_{n,z}$ defined by setting
	\begin{equation*}
		Y_{n,z}(\w)\defas\inf\Big\{E_{g}(\w)(u,A_{n,\nu}(z))\colon u\in \Adm(u_{y_n(z)}^{a,b,\nu},A_{n,\nu}(z))\Big\}
	\end{equation*}
	have the same distribution as $(t'_n)^{d-1}X_{t_n',\ell_n'}^{a,b,\nu}(g)$ for all $z\in I_{n,k}$ (and are measurable). Note that the boundary value has changed since in general $y_n(z)\notin H^{\nu}$. Let us number these random variables by numbering the finitely many elements in $I_{n,k}$,\ie we write $I_{n,k}=\{z^1,\ldots,z^r\}$  with $r=r(n,k)$. Since the cubes $Q_{n,z}$ are pairwise disjoint, it follows that
	\begin{equation}\label{eq:numberofsmallcubes}
		r\leq\left(\frac{t_k}{t'_n}\right)^{d-1}.
	\end{equation}
	For $i=1,\ldots,r$ let $u_{n}^i$ be a candidate for the optimization problem defining $Y_{n,z^i}(\w)$. We now define $v\in BV(R_{t_k,\ell_k}^\nu;\mathcal{M})$ by setting
	\begin{equation*}
		v(x)\defas
		\begin{cases}
			u_{n}^i(x)&\mbox{if $x\in A_{n,\nu}(z^i)\setminus \bigcup_{j=1}^{i-1}A_{n,\nu}(z^j)$ for some $i=1,\ldots,r$,}
			\\
			u^{a,b,\nu}(x) &\mbox{otherwise.}	
		\end{cases}
	\end{equation*}
	Thanks to~\eqref{eq:shiftfullinclusion} the function $v$ is admissible for the minimum problem defining $X_{t_k,\ell_k}^{a,b,\nu}(g)$.
	In order to estimate its energy, we split the jumpset into three different parts: the portion inside a set $A_{n,\nu}(z^i)$, on the boundary of some $A_{n,\nu}(z^i)$, and in the complement of $\bigcup_{i=1}^r \overline{A_{n,\nu}(z^i)}$. From the definition of $v$ we infer that
	\begin{align}\label{eq:splittingincases}
		E_{g}(\w)(v, R_{t_k,\ell_k}^{\nu})\leq& \sum_{i=1}^r E_{g}(\w)(u_{n}^i,A_{n,\nu}(z^i))\nonumber
		\\
		&+c\sum_{i=1}^r\mathcal{H}^{d-1}(\partial A_{n,\nu}(z^i)\cap S_v)+c\mathcal{H}^{d-1}\bigg(\Big(R_{t_k,\ell_k}^{\nu}\setminus\bigcup_{i=1}^r\overline{A_{n,\nu}(z^i)}\Big)\cap H^{\nu}\bigg).
	\end{align}
	We argue that the terms in the second line are asymptotically negligible. We start with the last term, which can be estimated using a purely geometrical argument. Note that when $x\in\left(R_{t_k,\ell_k}^{\nu}\setminus\bigcup_{i=1}^r\overline{A_{n,\nu}(z^i)}\right)\cap H^{\nu}$, then there are two exhaustive cases:
	\begin{align*}
		&i)\quad x\in t_k (Q^{\nu}\cap H^{\nu})\setminus \bigcup_{i=1}^r O_{\nu}(Q_{n,z^i}\times\{0\}),
		\\
		&ii)\quad x\in  O_\nu(Q_{n,z^i}\times\{0\})\setminus \overline{A_{n,\nu}(z^i)}\quad\text{for some }i.
	\end{align*}
	In the first case we can use \eqref{eq:almostfilling} and deduce that
	\begin{equation}\label{eq:firstcase}
		\mathcal{H}^{d-1}\left(t_k (Q^{\nu}\cap H^{\nu})\setminus \bigcup_{i=1}^r O_{\nu}(Q_{n,z^i}\times\{0\})\right)\leq t_k^{d-1}-(t_k-4\sqrt{d}t'_n)^{d-1}\leq C t_k^{d-2}t'_n.	
	\end{equation}
	In the second case, note that there exists a point $y$ on the segment $[0,y_n(z^i)]$ such that $x+y\in\partial A_{n,\nu}(z^i)=z_n(z^i)+\partial R_{t'_n,\ell'_n}^{\nu}$.
	In view of~\eqref{eq:decomposition} we have $z_n(z^i)-y_n(z^i)\in H^\nu$. Since also $x\in H^{\nu}$, we infer that
	\begin{equation*}
		|\langle x+y-z_n(z^i),\nu\rangle|=|\langle y-z_n(z^i),\nu\rangle|=|\langle y-y_n(z^i),\nu\rangle|\leq|\langle y_n(z^i),\nu\rangle|\leq\sqrt{d}.
	\end{equation*}
	Thus, for $t'_n$ large enough the condition $x+y-z_n(z^i)\in\partial R_{t'_n,\ell'_n}^{\nu}$ implies that there exists $j\in\{1,\ldots,d-1\}$ such that 
	\begin{equation*}
		|\langle x+y-z_n(z^i),O_{\nu}e_j\rangle|=\frac{t'_n}{2}.
	\end{equation*}
	Since $|y|\leq|y_n(z^i)|\leq\sqrt{d}$ this give
	\begin{equation*}
		\frac{t'_n}{2}-2\sqrt{d}\leq |\langle x-z_n(z^i),O_{\nu}e_j\rangle|\leq \frac{t'_n}{2}+2\sqrt{d}.
	\end{equation*}
	In particular, 
	\begin{equation*}
		\mathcal{H}^{d-1}(O_{\nu}(Q_{n,z^i}\times\{0\})\setminus \overline{A_{n,\nu}(z^i)})\leq C (t'_n)^{d-2}.
	\end{equation*}
	Taking into account the bound \eqref{eq:numberofsmallcubes}, we conclude that
	\begin{equation}\label{eq:secondcase}
		\sum_{i=1}^r\mathcal{H}^{d-1}(O_{\nu}(Q_{n,z^i}\times\{0\})\setminus \overline{A_{n,\nu}(z^i)})\leq C \frac{t_k^{d-1}}{t'_n}.
	\end{equation}
	Gathering \eqref{eq:firstcase} and \eqref{eq:secondcase} we finally obtain
	\begin{equation}\label{eq:firstandsecondcase}
		\mathcal{H}^{d-1}\bigg(\Big(R_{t_k,\ell_k}^{\nu}\setminus\bigcup_{i=1}^r\overline{A_{n,\nu}(z^i)}\Big)\cap H^{\nu}\bigg)\leq C \bigg(t_k^{d-2}t'_n+\frac{t_k^{d-1}}{t'_n}\bigg).
	\end{equation}
	Next we treat the term $\mathcal{H}^{d-1}(\partial A_{n,\nu}(z^i)\cap S_v)$. Consider $x\in\partial A_{n,\nu}(z^i)\cap S_v$ such that it is in the relative interior of $\partial A_{n,\nu}(z^i)$ (the measure of the remaining part is negligible). Then without loss generality we can assume that
	\begin{align*}
		&v^+(x)=u_i^+(x)=u^{a,b,\nu}_{y_n(z^i)}(x)=b,
		\\
		&v^-(x)= u^{a,b,\nu}_{y_n(z^j)}(x)=a\; \text{ for some }\; j\neq i\quad\text{or}\quad v^-(x)=u^{a,b,\nu}(x)=a.
	\end{align*} 
	This implies that
	\begin{align*}
		&\langle x-y_n(z^i),\nu\rangle >0,
		\\
		&\langle x-y_n(z^j),\nu\rangle \leq 0\quad\text{or}\quad \langle x,\nu\rangle \leq 0.
	\end{align*}
	If $\langle x,\nu\rangle\leq 0$, we have $0<\langle x-y_n(z^i),\nu\rangle\leq\langle-y_n(z^i),\nu\rangle\leq\sqrt{d}$. If instead $\langle x-y_n(z^j),\nu\rangle\leq 0$, we deduce that 
	\begin{equation*}
		0<\langle x-y_n(z^i),\nu\rangle=\langle x-y_n(z^j),\nu\rangle+\langle y_n(z^j)-y_n(z^i),\nu\rangle\leq \langle y_n(z^j)-y_n(z^i),\nu\rangle\leq 2\sqrt{d}.
	\end{equation*}
	Since $z_n(z^i)-y_n(z^i)\in H^{\nu}$, the above estimates yield as well
	\begin{equation*}
		0<\langle x-z_n(z^i),\nu\rangle\leq 2\sqrt{d}.
	\end{equation*}
	If instead $v^+(x)=a$ and $v^-(x)=b$, we deduce by the same argument that $-2\sqrt{d}\leq \langle x-z_n(z^i),\nu\rangle\leq 0$. Thus we obtain for $\ell'_n$ large enough that
	\begin{align*}
		\mathcal{H}^{d-1}(\partial A_{n,\nu}(z^i)\cap S_v)\leq& \mathcal{H}^{d-1}(\partial A_{n,\nu}(z^i)\cap \{x:\,|\langle x-z_n(z^i),\nu\rangle|\leq 2\sqrt{d}\})
		\\
		=&\mathcal{H}^{d-1}(\partial R_{t'_n,\ell'_n}^{\nu}\cap \{y:\,|\langle y,\nu\rangle|\leq 2\sqrt{d}\})\leq C (t'_n)^{d-2},
	\end{align*}
	where we used a change of variables taking into account that $\partial A_{n,\nu}(z^i)-z_n(z^i)=\partial R_{t'_n,\ell'_n}^{\nu}$. Summing over all $i=1,\ldots,r$ we deduce from \eqref{eq:numberofsmallcubes} that
	\begin{equation}\label{eq:otherterm}
		\sum_{i=1}^r\mathcal{H}^{d-1}(\partial A_{n,\nu}(z^i)\cap S_v)\leq C\frac{t_k^{d-1}}{t'_n}.	
	\end{equation}
	Inserting \eqref{eq:firstandsecondcase} and \eqref{eq:otherterm} in \eqref{eq:splittingincases} we infer that
	\begin{equation*}
		E_{g}(\w)(v,R_{t_k,\ell_k}^{\nu})\leq \sum_{i=1}^r E_{g}(\w)(u_{n}^i,A_{n,\nu}(z^i))+C\left( t_k^{d-2}t'_n+\frac{t_k^{d-1}}{t'_n}\right).
	\end{equation*}
	Note that thanks to~\eqref{eq:shiftfullinclusion}, $v$ is admissible for the minimization problem defining $X_{t_k,\ell_k}^{a,b,\nu}$. Thus, since $u_{n}^i(\w)$ was arbitrary by minimization and using that $Y_{n,z^i}$ have the same distribution as $(t'_n)^{d-1}X_{t'_n,\ell'_n}^{a,b,\nu}(g)$, we can take the expectation of the above estimate to deduce that
	\begin{align*}
		\mathbb{E}[X_{t_k,\ell_k}^{a,b,\nu}(g)]\leq& \frac{1}{t_k^{d-1}}\mathbb{E}[E_{g(\cdot)}(v,R_{t_k,\ell_k}^{\nu})]\leq \underbrace{r\left(\frac{t'_n}{t_k}\right)^{d-1}}_{\leq 1}\mathbb{E}[X_{t'_n,\ell'_n}^{a,b,\nu}(g)]+\frac{C}{t_k^{d-1}}\left( t_k^{d-2}t'_n+\frac{t_k^{d-1}}{t'_n}\right)
		\\
		\leq & \mathbb{E}[X_{t'_n,\ell'_n}^{a,b,\nu}(g)]+C\left( \frac{t'_n}{t_k}+\frac{1}{t'_n}\right).
	\end{align*}
	Now letting first $k\to +\infty$ and then $n\to +\infty$ we obtain
	\begin{equation*}
		\limsup_{k\to +\infty}\mathbb{E}[X_{t_k,\ell_k}^{a,b,\nu}(g)]\leq\liminf_{n\to +\infty}\mathbb{E}[X_{t'_n,\ell'_n}^{a,b,\nu}(g)].
	\end{equation*}
	Since the sequences $t'_n,\ell'_n$ and $t_k,\ell_k$ were arbitrary, the limit of the expectations exists. By setting $\ell_k=t_k$, we deduce the claim from Theorem \ref{t.compressedprocess} and the dominated convergence theorem.
\end{proof}
Now we can prove the almost sure convergence of the process $X_{t,\ell_t}^{a,b,\nu}(g)$ towards $g_{\rm hom}(a,b,\nu)$ under the Assumption \ref{ass:2} when the weight $\pi$ has higher integrability. 
\begin{proof}[Proof of Corollary \ref{c.asirrational}]
	Let us fix $\nu\in\Sd$, $(a,b)\in\mathcal{M}\times\mathcal{M}$, and let $\bar{\ell}_n\to+\infty$ be an arbitrary diverging sequence. We first show that almost surely we have
	\begin{equation}\label{eq:bar_l_n}
		\lim_{n\to +\infty}\left(X_{n,\bar{\ell}_n}^{a,b,\nu}(g)(\omega)-\mathbb{E}[X_{n,\bar{\ell}_n}^{a,b,\nu}(g)]\right)=0.
	\end{equation}
	To this end, let $r>2(d-1)$ be as in the assumptions and let $p\defas\frac{r}{2(d-1)}>1$, so that $2p(d-1)=r$. Moreover, let $\alpha>0$ be sufficiently small such that $p\,\alpha\in(0,p\,(d-1)-1)\neq\emptyset$; upon decreasing $\bar{\ell}_n$ and taking its lower integer part it is not restrictive to assume that $\bar{\ell}_n\leq n^\alpha$ and $\bar{\ell}_n\in\N$ for every $n\in\N$, so that
	\begin{equation}\label{cond:ell-n-bar}
		(\overline{\ell}_n n^{1-d})^p\leq n^{p\,\alpha-p\,(d-1)}.
	\end{equation} 
	The choice of $\alpha$ ensures that $p\,\alpha-p\,(d-1)<-1$. Hence, for every $\delta>0$ an application of Chebyshev's inequality together with Theorem~\ref{t.concentration} and~\eqref{cond:ell-n-bar} gives
	\begin{align*}
		\sum_{n\in\N}\mathbb{P}\left(|X_{n,\bar{\ell}_n}^{a,b,\nu}(g)(\omega)-\mathbb{E}[X_{n,\bar{\ell}_n}^{a,b,\nu}(g)]|\geq\delta\right)\leq& \sum_{n\in\N}\frac{1}{\delta^{2p}}\E\big[\big|X_{n,\bar{\ell}_n}^{a,b,\nu}(g)-\E[X_{n,\bar{\ell}_n}^{a,b,\nu}(g)]\big|^{2p}\big]
		\\
		\leq & C(\delta,p) \sum_{n\in\N}n^{p\,\alpha-p\,(d-1)} \int_0^{+\infty}(s+1)^{r}\pi(s)\,\mathrm{d}s<+\infty.
	\end{align*}
	Thus, the sequence $X_{n,\bar{\ell}_n}^{a,b,\nu}(g)(\omega)-\mathbb{E}[X_{n,\bar\ell_n}^{a,b,\nu}(g)]$ converges completely and hence almost surely to $0$,\ie~\eqref{eq:bar_l_n} follows. As a consequence, using Proposition~\ref{p.expectation} and Remark~\ref{r.mainresult}~\ref{rem:rational} we find a set $\Omega'\subset\Omega$ of full probability such that
	\begin{equation}\label{e:convergence-irrational-n}
	\lim_{n\to+\infty}X_{n,\bar{\ell}_n}^{a,b,\nu}(g)(\omega)=\lim_{n\to+\infty}X_{n,n}^{a,b,\nu}(g)(\omega)=g_{\rm hom}(a,b,\nu)\quad\text{for every}\ \w\in\Omega'.
	\end{equation}
	Note that for any $n\in\N$ and any $\ell\in[\bar{\ell}_n,n]$ the monotonicity property~\eqref{est:monotonicity-ell} implies that
	\begin{equation*} 
	X_{n,n}^{a,b,\nu}(g)(\omega)\leq X_{n,\ell}^{a,b,\nu}(g)(\omega)\leq X_{n,\bar{\ell}_n}^{a,b,\nu}(g)(\omega)\quad\text{for every}\ \w\in\Omega.
	\end{equation*}
	In particular, in view of~\eqref{e:convergence-irrational-n}, for any $\delta>0$ and any $\omega\in\Omega'$ there exists $n_0=n_0(\omega,\delta)\in\N$ such that
	\begin{equation*}
	|X_{n,\ell}^{a,b,\nu}(g)(\omega)-g_{\rm hom}(a,b,\nu)|<\delta\quad\text{for all}\ n\geq n_0\ \text{and all}\ \ell\in[\bar{\ell}_n,n].
	\end{equation*}
	From this we immediately deduce that
	\begin{equation}\label{ex:limit-tn-natural}
	\lim_{\substack{t_n\to+\infty\\t_n\in\N}}X_{t_n,\ell_n}^{a,b,\nu}(g)(\omega)= g_{\rm hom}(a,b,\nu)\quad\text{for every}\ \w\in\Omega',
	\end{equation}
	provided $t_n\geq\ell_n\geq\bar{\ell}_n$ for every $n\in\N$.
%	Next we consider a sequence $(\ell_n)$ such that $ \overline{\ell}_n\leq\ell_n\leq n$ for all $n\in\N$. Then by monotonicity we have
%	\begin{equation*}
%		X_{n,\overline{\ell}_n}^{a,b,\nu}(g)(\w)\geq X_{n,\ell_n}^{a,b,\nu}(g)(\w)\geq X_{n,n}^{a,b,\nu}(g)(\w).
%	\end{equation*}
%	The left hand side and the right hand side terms converge almost surely to $g_{\rm hom}(a,b,\nu)$, so that also $X_{n,\ell_n}^{a,b,\nu}(g)(\w)\to g_{\rm hom}(a,b,\nu)$ almost surely. 
	%The case of an arbitrary sequence $\ell_n\leq n$ such that $\ell_n\to +\infty$ can be analyzed by considering up to two subsequences converging to the same limit.
	
	The case of arbitrary sequences $t_n\to +\infty,\ell_n\to+\infty$ with $t_n\geq\ell_n\geq\bar{\ell}_n$ for every $n\in\N$ can be treated by combining~\eqref{est:monotonicity-ell} and~\eqref{est:monotonicity-t}. Namely, we consider the auxiliary sequences $t_n^-\defas\lfloor t_n\rfloor$, $t_n^+\defas\lceil t_n\rceil$. 
%	Since $\bar{\ell}_n\in\N$, we have
%	%
%	\begin{equation*}
%	\bar{\ell}_n\leq t_n^-\leq t_n\leq t_n^+,
%	\end{equation*}
%	where to obtain the first inequality we used that $\bar{\ell}_n\in\N$.
	Then~\eqref{est:monotonicity-t} yields
	\begin{align*}
		X_{t_n^+,\ell_n}^{a,b,\nu}(g)(\w)\leq X_{t_n,\ell_n}^{a,b,\nu}(g)(\w)+\frac{c(t_n^+-t_n)}{t_n^+}.
	\end{align*}
%	Since $t_n\to+\infty$, for any subsequence $(t_{n_k})$ we can extract a further subsequence $t_{n_{k_j}}$ such that
%	\begin{equation*}
%	t_{n_{k_j}}^-=j\quad\text{and}\quad t_{n_{k_j}}^+=j+1\quad\text{for every}\ j\in\N.
%	\end{equation*}
	Since $\ell_n\leq t_n^+$, we deduce from~\eqref{ex:limit-tn-natural} that
	\begin{equation*}
		\liminf_{n\to +\infty}X_{t_n,\ell_n}^{a,b,\nu}(g)(\w)\geq\liminf_{n\to+\infty}X_{t_n^+,\ell_n}^{a,b,\nu}(g)(\w)\geq g_{\rm hom}(a,b,\nu).
	\end{equation*}
	To prove the reverse inequality for the limit superior, we combine~\eqref{est:monotonicity-t} with~\eqref{est:monotonicity-ell} to deduce that
	\begin{equation*}
		X_{t_n,\ell_n}^{a,b,\nu}(g)(\w)\leq X_{t_n,\bar{\ell}_n}^{a,b,\nu}(g)(\w)\leq X_{t_n^-,\bar{\ell}_n}^{a,b,\nu}(g)(\w)+\frac{c(t_n-t_n^-)}{t_n}.
	\end{equation*}
	By assumption $\bar{\ell}_n\in\N$, so that $\bar{\ell}_n\leq t_n^-$. Thus, applying~\eqref{ex:limit-tn-natural} yields
	\begin{equation*}
		\limsup_{n\to +\infty}X_{t_n,\ell_n}^{a,b,\nu}(g)(\w)\leq\limsup_{n\to+\infty}X_{t_n^-,\bar{\ell}_n}^{a,b,\nu}(g)(\w)\leq g_{\rm hom}(a,b,\nu),
	\end{equation*} 
	which concludes the proof.
\end{proof}
\section*{Acknowledgments}
\noindent The authors wish to thank Marco Cicalese for fruitful discussions.
The work of A.\ Bach was supported by the DFG Collaborative Research Center TRR 109, ``Discretization in Geometry and Dynamics''.

\end{document}

%% file: uell.pdf_tex
%% Accompanies image file 'uell.pdf' (pdf, eps, ps)
\begingroup%
  \makeatletter%
  \providecommand\color[2][]{%
    \errmessage{(Inkscape) Color is used for the text in Inkscape, but the package 'color.sty' is not loaded}%
    \renewcommand\color[2][]{}%
  }%
  \providecommand\transparent[1]{%
    \errmessage{(Inkscape) Transparency is used (non-zero) for the text in Inkscape, but the package 'transparent.sty' is not loaded}%
    \renewcommand\transparent[1]{}%
  }%
  \providecommand\rotatebox[2]{#2}%
  \newcommand*\fsize{\dimexpr\f@size pt\relax}%
  \newcommand*\lineheight[1]{\fontsize{\fsize}{#1\fsize}\selectfont}%
  \ifx\svgwidth\undefined%
    \setlength{\unitlength}{55.12032757bp}%
    \ifx\svgscale\undefined%
      \relax%
    \else%
      \setlength{\unitlength}{\unitlength * \real{\svgscale}}%
    \fi%
  \else%
    \setlength{\unitlength}{\svgwidth}%
  \fi%
  \global\let\svgwidth\undefined%
  \global\let\svgscale\undefined%
  \makeatother%
  \begin{picture}(1,0.58689267)%
    \lineheight{1}%
    \setlength\tabcolsep{0pt}%
    \put(0,0){\includegraphics[width=\unitlength,page=1]{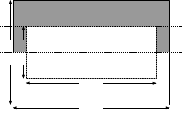}}%
    \put(0,0){\includegraphics[width=\unitlength,page=2]{uell1.pdf}}%
    \put(0.4665,0){\color[rgb]{0,0,0}\makebox(0,0)[lt]{\lineheight{1.25}\smash{\begin{tabular}[t]{l}{\small $t$}\end{tabular}}}}%
    \put(0.43,0.13){\color[rgb]{0,0,0}\makebox(0,0)[lt]{\lineheight{1.25}\smash{\begin{tabular}[t]{l}{\small $t-1$}\end{tabular}}}}%
    \put(0.4665,0.27){\color[rgb]{0,0,0}\makebox(0,0)[lt]{\lineheight{1.25}\smash{\begin{tabular}[t]{l}{\small $0$}\end{tabular}}}}%
    \put(0.065,0.29){\color[rgb]{0,0,0}\rotatebox{90}{\makebox(0,0)[lt]{\lineheight{1.25}\smash{\begin{tabular}[t]{l}{\small $2\ell$}\end{tabular}}}}}%
    \put(0.13,0.245){\color[rgb]{0,0,0}\rotatebox{90}{\makebox(0,0)[lt]{\lineheight{1.25}\smash{\begin{tabular}[t]{l}{\small $2i_\ell-1$}\end{tabular}}}}}%
    \put(0.96,0.3){\color[rgb]{0,0,0}\makebox(0,0)[lt]{\lineheight{1.25}\smash{\begin{tabular}[t]{l}{\small $\{x_2=0\}$}\end{tabular}}}}%
    \put(0.96,0.44){\color[rgb]{0,0,0}\makebox(0,0)[lt]{\lineheight{1.25}\smash{\begin{tabular}[t]{l}{\small $\{x_2=i_\ell-\tfrac{1}{2}\}$}\end{tabular}}}}%
  \end{picture}%
\endgroup%